\documentclass[a4paper, reqno, 12pt]{amsart}
\usepackage[english]{babel}
\usepackage[T1]{fontenc}
\usepackage{mathptmx,amsmath}
\usepackage{hyperref}
\usepackage{graphicx}
\usepackage{fullpage}
\newtheorem{theorem}{Theorem}[section]
\newtheorem{corollary}[theorem]{Corollary}
\newtheorem{lemma}[theorem]{Lemma}
\newtheorem{definition}[theorem]{Definition}
\newtheorem{proposition}[theorem]{Proposition}
 \theoremstyle{remark}
\newtheorem{remark}[theorem]{Remark}

\numberwithin{equation}{section}
\def\PSH{\text{PSH}}

\def\al{\alpha}
\def\la{\lambda}
\def\va{\varphi}
\def\red{\text{red}}
\def\reg{\text{reg}}
\def\su{\text{supp}}
\def\sig{\text{sig}}
\def\ep{\varepsilon}
\def\C{\mathbb C}
\def\R{\mathbb R}
\def\B{\mathbb B}
\def\loc{\text{loc}}
\def\Om{\Omega}
\def\CC{\mathcal C}
\def\fr{\frac}
\def\F{\mathcal F}
\def\ov{\overline}

\def\psh{plurisubharmonic}
\def\om{\omega}
\def\ve{\varepsilon}
\def\de{\delta}
\def\nhd{neighborhood}

\def\De{\Delta}
\def\E{\mathcal E}
\def\va{\varphi}

\def\B{\mathbb B}
\begin{document}
\title {Monge-Amp\`ere operators on complex varieties in $\C^n$}
\address{Department of Mathematics, Hanoi National University of Education,
136 Xuan Thuy street, Cau Giay, Hanoi, Vietnam}
\author{Nguyen Quang Dieu, Tang Van Long and Ounhean Sanphet}
\email{ngquangdieu@hn.vnn.vn, tvlong@hnue.edu.vn , sanphetMA@gmail.com }



\date{\today}

%\dedicatory{}
\maketitle
\section{Introduction}
Let $V$ be a connected complex variety of pure dimension $k$ in a domain $D \subset \mathbb C^ n (n  \ge 2, 1 \le k \le n-1).$ We denote by $V_{\reg}$ the set of regular points of $V.$ Thus $V_{\reg}$ is the largest (possibly disconnected) complex manifold of dimension $k$ included in $V.$ The singular locus of $V$ is then denoted by
$V_{\sig}:=V \setminus V_{\reg}.$ We also concern about the set of points near which $V$ is irreducible.
More precisely $V$ is said to be {\it locally irreducible} at $a \in V$ if there exists a
fundamental system of open  \nhd s $U_j$ of  $a$ such that each
$U_j\cap V$ is irreducible in $U_j$ (see [Ch] p.55).
We denote by $V_{\red}$ the set  of reducible points of  in $V$, i.e., $V_{\red}$ is the collection of points $a \in V$
such that $U \cap V$ is reducible in $U$ for {\it some} open \nhd\ $U$ of $a.$
It is then clear that $V_{\red}$ is a local (but not necessarily closed) complex subvariety of $V$.
For example, consider the Whitney umbrella variety $V^1:=\{(z_1, z_2, z_3) \in \mathbb B: z_1^2-z_2^2z_3=0\},$
where $\mathbb B$ is the unit ball in $\mathbb C^3,$
then $V^1_{\red}=\{(0,0,z_3): 0<\vert z_3\vert<1\},$ but $V^1$ is irreducible at the origin.
The case where $V$ is locally irreducible at any of its points (locally irreducible for short) or equivalently $V_{\red}=\emptyset,$ is of particular interest to us since in this case we may create plurisubharmonic functions by taking upper semicontinuous regularization of  families of {\it plurisubharmonic} functions which are locally bounded from above on $V.$ See Proposition \ref{regular} in the next section.

Recall that $u: V \to [-\infty, \infty), u \not\equiv-\infty$ on any irreducible component of $V,$ is said to be \psh\ if $u$ is locally the restriction (on $V$) of \psh\ functions on an open subset of $D$.
A fundamental result of Fornaess and Narasimhan (cf. Theorem 5.3.1 in [FN]) asserts that an upper semicontinuous function $u: V \to \R \cup [-\infty, \infty)$ which is not identically $-\infty$ on any irreducible component of $V$,
is \psh\  if and only if for every holomorphic map $\theta: \De \to V$, where $\De$ is the unit disk in $\C$, we have
$u \circ \theta$ is subharmonic on $\De.$ This powerful result implies immediately the nontrivial facts that  plurisubharmonicity is preserved under local uniform convergence.
We write $PSH (V)$ for the set of \psh\ functions on $V$ and $PSH^-(V)$ denotes  the subset of negative \psh\ functions on $V$.  The complex variety $V$ is said to be {\it hyperconvex} if there exists
$\rho \in PSH^-(V) \cap L^\infty(V)$ such that $\{z \in V: \rho(z) < c\}$ is relatively compact in $V$ for each $c<0.$
Note that $\rho$ is not assumed to be continuous on $V.$
Since $-1/\rho \in PSH(V)$ for every $\rho \in PSH^{-}(V)$, we see that every hyperconvex variety is Stein, i.e.,
there exists a \psh\ exhaustion function for $V.$

Next, we turn to the complex Monge-Amp\`ere operator for locally bounded \psh\ functions on $V.$
First, we note that by Proposition 1.8 in [Dem], each $\psi \in PSH(V)$ (not necessarily locally bounded) is locally integrable on $V$ with respect to the Lebesgue measure. Thus $dd^c \psi$ is a closed positive $(1,1)$ current on $V.$
According to Bedford in [Be] (see also [Dem]), the complex Monge-Amp\`ere operator
$$(dd^c )^k: PSH(V) \cap L^{\infty}_{\text{loc}} (V) \to M_{k,k}(V),$$
where $M_{k,k} (V)$ denotes Radon measures on $V,$ may be defined in the usual way on the regular locus $V_{\reg}$ of
$V$ (cf. [BT1],[BT2]), and it extends "by zero" through the singular locus $V_{\sig},$ i.e., for Borel sets
$E \subset V$
\begin{equation}\label{monge}
\int\limits_E (dd\psi)^k: =\int\limits_{E \cap V_\reg} (dd^c \psi)^k, \ \forall \psi \in
\PSH(V)\cap L^{\infty}_{\text{loc}}(V).
\end{equation}
The main goal of this article is to find, following the approach given in [Ce1] and [Ce2], the largest possible class
$\E (V)$ of $\PSH^{-} (V)$ on which the complex Monge-Amp\`ere operator can be reasonably defined. Moreover, we are also interested in giving sufficient condition for solvability of the Monge-Amp\`ere operator on a certain sub-class of
$\E (V)$. There are at least two technical difficulties in the process of generalizing Cegrell's machinery from the case of domains in $\mathbb C^n$ to the case of complex varieties $V$. Firstly, semi-local smoothing of \psh\ functions by taking convolutions with approximation of identities does not work on $V$, and secondly the upper-regularization of a family of \psh\ functions on $V$ may fail to be \psh\ if $V$ is not locally irreducible.

We now briefly outline the content of the paper. We first collect basics facts about plurisubharmonic functions and the complex Monge-Amp\`ere operators on complex varieties. The major tools are integration by parts formulas which are established by using a smoothing method devised by Bedford in [Be]. Being equipped with these formulas, we go on to prove some estimates for energy of \psh\ functions with zero boundary values.
Our main results are explained in Section 3. We first introduce in Proposition \ref{e0}, a class of hyperconvex varieties $V$ for which $\E_0 (V) \ne \emptyset$, i.e., there exists a bounded element in $\PSH^{-}(V)$ having zero boundary values and total {\it finite} Monge-Amp\`ere mass. It does not seem easy to decide if such function $u$ exists on a simple one dimension variety such as $V^2:=\bigcup\limits_{j \ge 1}\{(z_1,z_2)\in \mathbb B: z_1(z_2-a_j)=0\},$ where $\mathbb B$ is the unit ball in $\mathbb C^2$ and $\{a_j\}$ is a sequence in the unit disk $\De$ of
$\mathbb C$ such that $\vert a_j\vert \uparrow 1$ as $j \to \infty.$
The class $\E_0 (V)$ is then used to define and characterize $\E(V)$ the "largest" possible sub-class of $\PSH^{-} (V)$ on which the operator $(dd^c)^k$ can be defined and retains some basic properties as describe in [BT2], [Ce1] and [Ce2].
Our last main result is Theorem \ref{Cegrelltheorem} that ensures the existence in certain sub-class of $\E (V)$ of the unique solution to the Monge-Amp\`ere equation on $V.$
This result is achieved by following the scheme outlined by Cegrell in [Ce3]
(for the case of domains in $\mathbb C^n$) together with all the machinery we develop upto the theorem.
\vskip0,6cm
\noindent
{\bf Acknowledgments.} We are grateful to an anonymous referee for his(her) criticisms to an earlier version of this note. This work is written during a visit of the first named author to the Vietnam Institute for Advances in Mathematics (VIASM) in the winter of 2016. We would like to thank this institution for hospitality and financial support. Our work is also supported by the grant 101.02-2016.07 from the NAFOSTED program.
\section{Preliminary}
We first review some notions of pluripotential theory on complex varieties pertaining to our work.
For detailed discussions, the reader may consult the sources like [Be] and [Dem].
A subset $X \subset V$ is said to be {\it (locally) pluripolar} if for every $a \in V$, there exists a \nhd\ $U$ of $a$
and $v \in PSH(U)$ such that $v|_{X \cap U}=-\infty.$
For instance, $V_{\sig}$ being a {\it proper} closed complex subvariety of $V$ is (locally) pluripolar in $V.$
According to Theorem 5.3 in [Be], for every pluripolar subset $X$ of $V$, there exists $u \in PSH(V)$ such that
$u|_X \equiv -\infty.$ Of course, in the case of domains in $\mathbb C^n,$ this statement is contained in a classical theorem of Josefson.

To decide pluripolarity of subsets in $V$, following [Be] (see also [BT2] for the case of open domains in
$\mathbb C^n$), we define
for each Borel subset $E$ of an open set $\Om \subset V,$ the capacity of $E$ relative to $\Om$
\begin{equation} \label{capacity}
C(E,\Om):=\sup\big\{\int_E (dd^c u)^k: u \in \PSH(\Om): -1 \le u<0\big\}.
\end{equation}
The above definition makes sense since $\psi (z):=\fr{\Vert z\Vert^2}M-1 \in PSH(V)$ and satisfies $-1\le \psi<0,$
where $M:=\sup \{\vert z\vert^2: z \in D\}<\infty.$
In the case $\Om=V$, by an abuse of notation we will write $C(E)$ instead of $C(E,V).$

Obviously, by (\ref{capacity}), the singular locus $V_{\sig}$ has zero capacity, i.e.,
$C(V_{\sig} \cap \Om, \Om)=0$ for every open subset $\Om$ of $V.$
The following basic result of Bedford (Lemma 3.1 in [Be]) asserts that $V_{\sig}$ actually has {\it outer} capacity zero.
\begin{lemma}\label{outer} For every open subset $\Om$ of $V$ and
every $\ve>0$, there exists an open \nhd\ $U$ of $V_{\sig}$ in $\Om$ such that
$C(U,\Om)<\ve.$
\end{lemma}
\noindent
Lemma \ref{outer} implies that for each  $u \in PSH(V) \cap L^{\infty}_{\text{loc}} (V)$, (\ref{monge}) defines
$(dd^c u)^k$ as a Radon measure on $V$ which has no mass on $V_\sig.$
Moreover, if $u_1, \cdots, u_m \in PSH(V) \cap L^{\infty}_{\text{loc}} (V) (1 \le m \le k)$
then $dd^c u_1 \wedge \cdots \wedge dd^c u_m$ is a $(m,m)-$positive current on $V$ which is closed on $V_\reg$
and has no mass on $V_\sig.$
We also see that all local analysis for Monge-Amp\`ere operator in the class of {\it locally bounded} \psh\ functions in [BT2] carried over $V_\sig.$ For instance, $(dd^c)^k$ is continuous with respect to monotone convergence of locally uniformly bounded sequences in $PSH(V)$ and that every $u \in \PSH(V)$ is quasi-continuous, i.e., $u$ is continuous on the complement of open sets with arbitrarily small capacity.

A sequence of positive Borel measures $\{\mu_j\}$ on $V$ is said to have the {\it absolutely continuous with respect to capacity} (ACC for short) property if for every $\ve>0$, there exits $\de>0$ such that for each Borel set $E$ with $C(E)<\de$ we have
$\varlimsup\limits_{j \to \infty} \mu_j (E)<\ve.$
A subset $X$ of $V$ is said to be {\it quasi-open} if for every $\ve>0$ there exist an open set set $X_\ve \subset V$ such that $C(X_\ve \setminus X)<\ve, C(X \setminus X_\ve)<\ve.$
By the above quasi-continuity of \psh\ functions on $V$, we infer that $\{u<v\}$ is quasi-open for every $u,v \in \PSH(V)$.

The main advantage of the ACC property lies in the following convergence result.
\begin{lemma} \label{ACC}
Let $\{\mu_j\}, \mu$ be positive Borel measures on $V$ such that $\{\mu_j\}$ converges weak$^*$to $\mu$.
Assume that $\{\mu_j\}$ has the ACC property. Then for each Borel quasi-open set $X \subset V$ we have
$\lim\limits_{j\to \infty} \mu_j(X)=\mu(X).$
\end{lemma}
\begin{proof}
Fix $\ve>0.$ Then there exists $\de>0$ such that if $C(E)<\de$ then $\varlimsup\limits_{j \to \infty} \mu_j (E)<\ve.$
Choose a compact $K_\ve \subset X$ and an open set $X_\ve$ such that $C(X\setminus K_\ve)<\de, C(X_\ve \setminus X)<\ve, C(X\setminus X_\ve)<\ve.$
It follows that
$$\varlimsup\limits_{j\to \infty} \mu_j (X) \le
\varlimsup\limits_{j\to \infty} \mu_j (K_\ve)+\ve \le \mu(K_\ve)+\ve \le \mu(X)+\ve,$$
and
$$\varliminf\limits_{j\to \infty} \mu_j (X) \ge \varliminf\limits_{j\to \infty} \mu_j (X_\ve)-\ve \ge
\mu(X_\ve)-\ve \ge \mu(X)-2\ve.$$
We finish the proof by letting $\ve \downarrow 0.$
\end{proof}
A sequence $\{T_j\}$ of $(m,m)-$currents of order zero on $V$ is said to be convergent to a $(m,m)-$current $T$ if
$\langle T_j, \va\rangle \to \langle T,\va\rangle$ for every $(k-m,k-m)-$form $\va$ with coefficients in
$\mathcal C_0 (V)$. In the case where all $T_j$ are {\it positive} currents such that $T_j\wedge \om^{k-m}$, where $\om$ is the restriction on $V$ of the standard K\"ahler form $dd^c \vert z\vert^2$, put no mass on
$V_\sig$, it suffices to consider only $(k-m,k-m)-$forms $\va$ with coefficients are restriction on $V_\reg$ of functions in $\mathcal C^\infty_0 (D).$

For each subset $E \subset V$, the relative extremal function for $E$ in $V$ is defined as
$$u_{E,V}(z):= \sup\{u(z): u \in \PSH^{-}(V), u|_E\le -1\}.$$
This kind of functions will serve as contenders for evaluating the relative capacities.
We will write $u_E$ instead of $u_{E,V}$ if it is clear from the context.
A subset $E \subset  V$ is called  regular if the upper-regularization of $u_E$ which is defined as
$$u^*_E(z):= \varlimsup\limits_{x \to z} u_E (x), \forall z \in V,$$
is $\equiv -1$ on  $E$. In Lemma \ref{pro2} we will show that a compact  subset $K \subset V$ is regular if and only if $u^*_K$ is continuous on $V$.

To handle Monge-Amp\`ere operator on $PSH(V) \cap L^\infty_{\loc} (V)$, it is useful to recall the following quasi-smoothing process devised by Bedford in [Be].
Given $u \in PSH(V)$, we choose an open covering $\{U_j\}$ of $V$ such that for each $j$ there exists
an open set $\tilde U_j \subset D, U_j$ is a complex subvariety of $\tilde U_j,$ and there exists $u_j \in PSH(\tilde U_j)$ such that $u_j=u$ on $U_j.$ Let $\{\chi_j\}$ be a partition of unity subordinate to $\{\tilde U_j\}$ and
$\{\rho_\ve\}$ be standard radial smoothing kernels with compact support in the balls
$\mathbb B(0,\ve) \subset \mathbb C^k$. Now for $\ve>0$ small enough, our smoothing is obtained as the sum
$$u^\ve:= \sum_j \chi_j (u_j*\rho_\ve), z \in \mathbb C^n.$$
where $V_\ve:= \cup_j \{z \in U_j: \text{dist}\ (z, \partial \tilde U_j)>\ve\}.$
It is clear that $u^\ve \in \mathcal C^\infty (\mathbb C^n)$ with compact support in $D$. In particular, $u^\ve=0$
on a \nhd\ of $\partial V$ for each $\ve>0$ small enough.
Moreover, $u^\ve \downarrow u$ on $V$ as $\ve \downarrow 0$.
It should also be warned that $u^\ve$ may not be \psh\ on open subsets of $V$. We will see below, however, that
this smoothing process gives nice continuity property of the Monge-Amp\`ere operator.
Namely, we have the following approximation result which is already contained in [Be] (see also Proposition 2.3 in [DS])
\begin{proposition}\label{approximationma}
Let $\{f_j\}, f$ be locally uniformly bounded, quasi-continuous functions on $V$. Assume that $\{f_j\}$ converges locally
quasi-uniformly to $f.$
Then for every sequence $\{\de_j\}\downarrow 0,$
the currents $f_jdd^c u_1^{\de_j} \wedge \cdots\wedge dd^c u_k^{\de_j}$ converges weakly to $fdd^c u_1 \wedge \cdots \wedge dd^c u_k$ as $j \to \infty.$
\end{proposition}
Our next major tool is the following comparison principle that was essentially proved in [Be] by using
the above quasi-smoothing process together with a clever application of Stoke theorem.
\begin{theorem}\label{thmBe}
Let $u,v \in PSH(V) \cap  L^\infty(V)$ be  such that
$\varliminf\limits_{z \to \partial V} (u(z)-v(z))\ge 0.$ Then
$$\int\limits_{\{u<v\}}(dd^cv)^k \le \int\limits_{\{u<v\}}(dd^cu)^k.$$
\end{theorem}
\begin{proof}
Fix $\ve>0$. Set $u_\ve:=u+\ve.$ Then $\{u_\ve\le v\} \subset \subset V$.
In view of Theorem 4.3 in [Be] we get
$$\int\limits_{\{u_\ve<v\}}(dd^cv)^k \le \int\limits_{\{u_\ve<v\}}(dd^cu)^k \le \int\limits_{\{u<v\}}(dd^cu)^k.$$
By letting $\ve \downarrow 0$ and noting that $\{u+\ve<v\} \uparrow \{u<v\}$ we complete the proof.
\end{proof}
From the above theorem, in the same fashion as Corollary 4.3 in [BT2], we obtain the following estimates about total Monge-Amp\`ere masses of {\it bounded} elements in $PSH(V)$ having zero boundary values.
\begin{corollary}\label{corBe}
Let $u, v \in PSH^{-}(V) \cap L^\infty(V)$ be such that $u \le v$ on $V$ and
$\lim\limits_{z \to \partial  V} u(z) = \lim\limits_{z \to \partial V} v(z)=0$.
Then
$$\int\limits_V (dd^cv)^k  \le \int\limits_V(dd^c u)^k.$$
\end{corollary}
\begin{proof}
Fix $\lambda>1$. Then $\la u<v$ on $V.$ Applying Theorem \ref{thmBe} to
$\la u$ and $v$ we get
$$\int\limits_V(dd^cv)^k \le \la^k\int\limits_V (dd^cu)^k.$$
By letting $\la \downarrow 1,$ we obtain the desired inequality.
\end{proof}
\noindent
We also need the following version of the integration by part formula.
\begin{proposition} \label{partformular1} Let $u, v, w_1, \cdots, w_k \in PSH(V)\cap L^\infty_{\loc} (V).$
Assume that $u=v$ on a small neighborhood of  $\partial  V$. Then
$$\begin{aligned}
\int\limits_{V}(v-u)dd^c w_1\wedge \cdots \wedge dd^c w_k
&=\int\limits_{V} w_1 dd^c (v-u) \wedge dd^c w_2\wedge \cdots \wedge dd^c w_k.\\
\end{aligned}$$
In particular we have
\begin{equation} \label{equal}
\int\limits_{V}dd^c (v-u) \wedge dd^c w_2\wedge \cdots \wedge dd^c w_k=\int\limits_V[(dd^c u)^k-(dd^c v)^k]=0.
\end{equation}
\end{proposition}
\noindent
\begin{proof}
For $\de>0$ small enough,
we let $u^\de, v^\de, w_1^\de, \cdots, w_k^\de$ be smoothing of
$u, v, w_1, \cdots, w_k,$ respectively. Notice that the covering $\{U_j\}$ and the partition of unity $\{\chi_j\}$ can be chosen to be common for all these \psh\ functions. By the assumption we have $u^\de=v^\de$ outside a fixed compact  subset  $K$ of $V$ for every $\de>0$ small enough.
To simplify the notation, we set
$$T^\de:=dd^c w_2^\de \wedge \cdots \wedge dd^c w_k^\de, T:=dd^c w_2 \wedge \cdots \wedge dd^c w_k.
$$
Then an application of Stoke's theorem for smooth forms
on the complex variety $V$ (see p.33 in [Ch]) yields
$$\begin{aligned}
&\int_{V} \big [(v^\de -u^\de) dd^c w_1^\de -w_1^\de dd^c( v^\de -u^\de) \big] \wedge T^\de\\
&=\int_{V} d\big [(v^\de -u^\de) d^c w_1^\de \wedge T^\de -w_1^\de d^c (v^\de -u^\de)\wedge T^\de\big]=0.
\end{aligned}
$$
It follows that
\begin{equation} \label{Stoke}
\int\limits_{V} (v^\de -u^\de) dd^c w_1^\de \wedge T^\de
=\int\limits_{V} w_1^\de dd^c (v^\de -u^\de) \wedge T^\de.
\end{equation}
We now consider the limits of both sides of (\ref{Stoke}) as $\de \downarrow 0.$ For the left hand side, set
$$\mu^\de:= (v^\de -u^\de) dd^c w_1^\de\wedge T^\de,
\mu:= (v-u)dd^c w_1 \wedge T.$$
Notice that
$v^\de-u^\de$ are  continuous functions on $V$ that converges to $(v-u)$ locally quasi uniformly on $V$.
Hence, by Proposition \ref{approximationma}  we deduce that $\mu^\de$ converge weakly to $\mu$ as $\de \downarrow 0.$ Moreover,  $\mu^\de$ and $\mu$ are real measures on $V$ that vanish outside $K$.
Let $f$ be a continuous function with compact support in $V$ such that $0 \le f \le 1$ and $f|_K=1.$
Then we have
$$\mu^\de (V)=\mu^\de (K)=\int_V fd\mu^\de \to \int_V fd\mu=\mu(V) \ \text{as}\ \de \to 0.$$
It means that
$$\lim\limits_{\de \to 0} \int_V(v^\de -u^\de) dd^c w_1^\de\wedge T^\de \to \int_V  ( v-u) dd^c w_1 \wedge T.$$
By a similar reasoning, we also have
$$\lim\limits_{\de \to 0} \int\limits_V w_1^\de  dd^c (v^\de -u^\de)\wedge T^\de \to \int_V  w_1 dd^c (v -u) \wedge T.$$
Putting all this together, we get the first assertion of the lemma. Finally, by taking $w_1 \equiv 1$ we obtain
the first equality in (\ref{equal}). This equality implies the other one by writing
$$\int\limits_V[(dd^c u)^k-(dd^c v)^k]=
\sum\limits_j \int\limits_V dd^c (u-v) \wedge (dd^c u)^j \wedge (dd^c v)^{k-1-j}=0.$$
The proof is thereby completed.
\end{proof}
The above result yields our first technical tool which essentially reduces to Theorem 3.2 in [Ce2]
in the case of domains in $\mathbb C^n$.
\begin{corollary}\label{coropartformular1}
Let $u, v, w_1,\cdots, w_{k-1} \in PSH^{-}(V)\cap L^\infty_{\loc} (V).$ Set
$T: =dd^c w_1 \wedge \cdots \wedge dd^c w_{k-1}$.
Assume that $\lim\limits_{z \to \partial V} u(z) = 0.$
Then we have
$$\int_V v dd^c u\wedge T \le \int_V u dd^c v\wedge T.$$
Consequently, if $\lim\limits_{z \to \partial V} v(z)=0$ then
$$\int_V v dd^c u\wedge T=\int_V u dd^c v\wedge T.$$
\end{corollary}
\begin{proof}
For  $\ep>0$ we  set $u_\ep:= \max \{u, -\ep\}$. Then  $u_\ep \in \PSH^{-} (V)  \cap  L_{\loc}^\infty (V), u_\ep=u$
near  $\partial V$ and  $(u- u_\ep)\downarrow  u$ on  $V$  as  $\ep \downarrow  0$.  By Proposition \ref{partformular1}   we obtain
$$\begin{aligned}
\int_V (u-u_\ep) dd^c v\wedge T
&= \int_V vdd^c ( u-u_\ep)  \wedge T\\
& = \int_V vdd^c u \wedge T - \int_V  v dd^c  u_\ep  \wedge T \ge \int_V  v dd^c  u  \wedge T.
\end{aligned}$$
By letting $\ep \downarrow 0$ and using Lebesgue's monotone convergence theorem we obtain the desired conclusion.
\end{proof}
\noindent
The following variant of the above integration by part formula is surprisingly harder to show. In the case of domains in $\C^n$, this result is implicitly included in [CP]. It seems to us that the method in [CP] does not directly extend to the case at hand.
\begin{corollary}\label{partformular2}
Let $u, v, v_1, ..., v_{k-1}\in PSH^{-}(V)\cap L^\infty(V)$ be such that
$\lim\limits_{z \to \partial V} u(z)=\lim\limits_{z \to \partial V} v(z)= 0$ and that
$$\int_V dd^c u \wedge T<\infty, \int_V dd^c v\wedge T<\infty,$$
where $T:=dd^c v_1 \wedge \cdots \wedge dd^c v_{k-1}$.
Then we have
$$\int\limits_V (-u)dd^c v \wedge T=\int_V du\wedge d^c v\wedge T<\infty.$$
\end{corollary}
\begin{proof}
For each $\ve > 0$, we set $u_\ve:= \{u, -\ve\}$. Then
$u_\ve=u$ on a small \nhd\ of $\partial V.$
Using the equation (\ref{equal}) in Proposition \ref{partformular1} we obtain
\begin{equation} \label{uep}
\int_V dd^c u_\ve \wedge T=\int_V dd^c u \wedge T=M<\infty, \ \forall \ve>0.
\end{equation}
We now prove the following assertions:

\noindent
(i) $\sup\limits_{\ve >0}\int\limits_V du_\ve \wedge d^c u_\ve \wedge T<\infty.$

\noindent
(ii) $\lim\limits_{\ve \to 0}\int\limits_V du_\ve \wedge d^c u_\ve \wedge T=0.$

For (i), we fix $\ve>0$. Then for each $\ve'>0$ we have $u_{\ve'}=u_\ve=u$ on a small \nhd\ of $\partial V.$
More precisely
\begin{equation}
u_\ve=u \ \text{on the closed set}\ \{u \ge -\ve'\}\  \text{for}\ 0<\ve'\le \ve.
\end{equation}
Then we have
$$\begin{aligned}
\int_V d(u_{\ve'}-u_\ve) \wedge d^c (u_{\ve'}-u_\ve) \wedge T&=\int_V (u_\ve-u_{\ve'}) dd^c (u_{\ve'}-u_\ve) \wedge T \\
&=\int_V u_\ve dd^c u_{\ve'} \wedge T+\int_V u_{\ve'} dd^c u_\ve \wedge T-\int_V u_{\ve'} dd^c u_{\ve'} \wedge T-
\int_V u_\ve dd^c u_\ve \wedge T\\
&<2\ve M,
\end{aligned}$$
where the last inequality follows from (\ref{uep}).
It then follows from the weak$^*-$ convergence of positive measures $d(u_{\ve'}-u_\ve) \wedge d^c (u_{\ve'}-u_\ve) \wedge T$ towards
$du_\ve \wedge d^c u_\ve \wedge T$ as $\ve' \downarrow 0$ that
$$\int\limits_V du_\ve \wedge d^c u_\ve \wedge T\le \varliminf \limits_{\ve' \to 0} \int_V d(u_{\ve'}-u_\ve) \wedge
d^c (u_{\ve'}-u_\ve) \wedge T
\le 2\ve M.$$
Hence we obtain (i) and (ii). In particular (i) implies that
$\int\limits_V du \wedge d^c u \wedge T<\infty$ and similarly $\int\limits_V dv \wedge d^c v \wedge T<\infty.$
Thus, in view of Cauchy-Schwarz's inequality we also get
$$\big \vert \int_V du \wedge d^c v \wedge T \big \vert^2 \le \big (\int_V du \wedge d^c u \wedge T \big)
\big (\int_V dv \wedge d^c v \wedge T\big )<\infty.$$
Now we turn to the proof of the lemma. For $\ve>0$ we have
$$\begin{aligned}
\int_V (u_\ve-u)dd^c v \wedge T &=\int_V d(u-u_\ve)\wedge d^c v \wedge T\\
&=\int_V du \wedge d^c v\wedge T-\int_V du_\ve \wedge d^c v \wedge T.
\end{aligned}$$
Using (ii) and the Cauchy-Schwarz inequality again we obtain
$$\big \vert \int_V du_\ve \wedge d^c v \wedge T \big \vert^2 \le \big (\int_V du_\ve \wedge d^c u_\ve \wedge T \big)
\big (\int_V dv \wedge d^c v \wedge T\big ) \to 0 \ \text{as}\ \ve \to 0.$$
This implies that
$$\int_V (-u)dd^c v \wedge T=\lim_{\ve \to 0}\int_V (u_\ve-u)dd^c v \wedge T=\int_V du \wedge d^c v \wedge T.$$
Here the first equality follows from Lebesgue's monotone convergence theorem. The proof is thereby completed.
\end{proof}
Next, following Cegrell in [Ce1], we define the following important subclass of $PSH^{-}(V)$
$$\E_0 (V):=  \Big \{u \in \PSH^{-} (V) \cap L^\infty (V):
\lim\limits_{z \to \partial V} u(z)=0, \int\limits_V (dd^c u)^k<\infty\Big \}.$$
\noindent
It is not clear to us that $\E_0 (V)$ is always {\it non-empty.} Nevertheless, we are able to show, in the next section, that if $V$ is locally irreducible then $\E_0 (V) \cap \mathcal C(V) \ne \emptyset.$
We will use the preceding integration by part formulas to get an energy estimate and also H\"older type estimates for Monge-Amp\`ere measures in $\E_0 (V)$. In the case of domains in $\C^n,$ the result below are included in Lemma 5.2, Theorem 3.4 of [Pe] and Theorem 5.5 in [Ce2].
\begin{lemma} \label{energy}
Let $u_0,u_1,..., u_k \in \E_0 (V).$ Then the following assertions hold true:

\noindent
(a) $\int\limits_V (-u_0) dd^cu_1\wedge T \le \Big(\int\limits_V(-u_0) dd^c u_0\wedge T\Big)^{\frac12} \Big(\int\limits_V(-u_1)dd^c u_1\wedge T \Big)^{\frac12},$

\noindent
where $T:=dd^c u_2 \wedge \cdots \wedge dd^c u_k.$

\noindent
(b) $\int\limits_V (-u_0) dd^cu_1\wedge T \le \Big(\int\limits_V(-u_0) (dd^c u_0)^k \Big )^{\fr1{k+1}} \cdots \Big(\int\limits_V(-u_k) (dd^c u_k)^k \Big )^{\fr1{k+1}}.$

\noindent
(c) $\int\limits_V(-u_0) dd^cu_1\wedge T \le \Big (\int\limits_V (-u_0)(dd^c u_1)^k\Big)^{1/k} \cdots
\Big (\int\limits_V(-u_0)(dd^c u_k)^k\Big)^{1/k}.$
\end{lemma}
\begin{proof}
(a) We proceed as in Lemma 5.2 in [Pe].
Using the integration by part formula
(cf. Corollary \ref{partformular2}) and Cauchy-Schwarz's inequality we obtain
$$\begin{aligned}
\int_V(-u_0)  dd^cu_1\wedge T &=  \int_V du_0\wedge d^cu_1\wedge T\\
&\le \Big(  \int_V du_0\wedge d^cu_0\wedge T\Big)^{\frac12} \Big(\int_V du_1\wedge d^cu_1\wedge T\Big)^{\frac12}\\
&= \Big(\int_V(-u_0) dd^c u_0\wedge T\Big)^{\frac12} \Big(\int_V(-u_1) dd^cu_1\wedge T\Big)^{\frac12}.
\end{aligned}$$
This is the desired estimate.

\noindent
(b) follows directly by applying (a) and Theorem 4.1 in [Pe] to the auxiliary function
$$F(u_0, u_1,\cdots u_k):=\int_V (-u_0) dd^c u_1 \wedge \cdots \wedge dd^c u_k.$$
\noindent
(c) Set $T':=dd^c u_3 \wedge \cdots \wedge dd^c u_k.$
Then using the integration by part formula
(cf. Corollary \ref{coropartformular1}) and (a) we obtain
$$\begin{aligned}
\int_V (-u_0)dd^cu_1\wedge dd^cu_2 \wedge T' &=  \int_V (-u_1)dd^c u_2\wedge dd^cu_0\wedge T'\\
&\le \Big(\int_V(-u_1) dd^c u_1 \wedge dd^cu_0 \wedge dd^c u_2\wedge T'\Big)^{\frac12}
\Big(\int_V(-u_2) dd^cu_2\wedge dd^c u_0 \wedge T'\Big)^{\frac12}\\
&= \Big(\int_V(-u_0) (dd^c u_1)^2 \wedge dd^c u_2\wedge T'\Big)^{\frac12}
\Big(\int_V(-u_0) (dd^cu_2)^2\wedge T'\Big)^{\frac12}.
\end{aligned}$$
Thus, so far we have followed the first step in the proof of Lemma 5.4 in [Ce2]. Nevertheless, to avoid Cegrell's somewhat complicated induction arguments (in the rest of Lemma 5.4 and in Theorem 5.5 of [Ce2]) we
invoke again Theorem 4.1 in [Pe], where the function $F$ is the same as in (b), but this time, the variable $u_0$
is regarded as a {\it fixed} parameter.
\end{proof}
\noindent
The following lemma about gluing \psh\ functions on $V$ is an easy consequence of the above mention Fornaess-Narasimhan's criterion for membership in $\PSH(V)$. To the best of our knowledge, no proof directly from the definition is known.
\begin{lemma} \label{gluing}
Let $U \subset V$ be an open subset and $u \in \PSH(V), v \in \PSH(U)$.
Assume that $\varlimsup\limits_{\xi \to z} v(\xi) \le u(z) \  \forall z \in \partial U.$ Then the function
$$w:=\begin{cases}
\max \{u,v\}  &  \text{on}\  U\\
u &  \text{on}\  V \setminus U.
\end{cases}$$
belongs to $PSH(V).$
\end{lemma}
\noindent
This lemma will facilitate some "spherical" modifications (analogous to the Poisson modifications with respect to the classical Dirichlet problem) in the next sections. We should say that this fact has been used implicitly in Theorem 1.8 of [Wik] and also in Lemma 5.1 and Theorem 5.3 of [Be].

Another technical tool comes from the following fact about upper-regularizations of subsets of $\PSH^{-}(V).$
\begin{proposition} \label{regular}
Let $\mathcal A \subset \PSH^{-} (V).$ Define
$$u(z):=\sup\{v(z): v \in \mathcal A\}, z \in V.$$
Assume that $V$ is locally irreducible at $x \in V.$ Then $u^*$ is \psh\ on some \nhd\ $U$ of $x.$
\end{proposition}
The above result is basically Theorem 1.5 in [Wik2] which is an easy consequence of an extension result (Theorem 1.7 in [Dem]) for \psh\ functions defined outside a nowhere dense closed complex subvariety $V'$ of $V$ which are locally upper bounded near points of $V'.$
It should be noticed that the assumption of local irreducibility of $V$ given in Theorem 1.5 of [Wik2] should be understood as locally irreducible at {\it every} point of $V.$
Using Proposition \ref{regular}, we may generalize the above result of Demailly as follows
\begin{proposition} \label{remove}
Let $X$ be a closed pluripolar subset of $V$ and $u \in \PSH(V \setminus X).$ Assume that $V$ is locally irreducible
and $u$ is locally bounded from above near every point of $X.$ Then there exists $\tilde u \in PSH(V)$ such that $\tilde u|_{V \setminus X}=u.$
\end{proposition}
\begin{proof} It suffices to prove that $u$ extends to a \psh\ function through each point of $X.$
Fix $z_0 \in X.$ Then there exist an open \nhd\ $U_{z_0} \subset V$ of $z_0$ and $\va \in PSH(U)$ such that
$\va|_{X \cap U} \equiv -\infty.$ By subtracting large constant, we may assume
$\va<0$ on $U.$ For $\ve>0,$ we set
$u_\ve: =u+\ve \va$ on $U \setminus X$ and $u_\ve:=-\infty$ on $U \cap X.$ Using Fornaess-Narasimhan's criterion we see that $u_\ve \in PSH^-(U).$ Thus, by Proposition \ref{regular}, we infer that
$$u_{z_0} (z):=(\sup \{u_\ve(z): \ve>0\})^*\in PSH^{-}(U),\forall z \in U.$$
It is also clear that $u_{z_0}=u$ on $U\setminus X.$ The proof is thereby completed.
\end{proof}
\noindent
We end up this section by sorting out the following extension result for \psh\ functions on $V.$
\begin{proposition}\label{extension}
Assume that $V$ is a Stein variety. Then for every $u \in PSH(V),$ there exists an open Stein \nhd\ $U$ of $V$ (that may depend on $u$) and $\tilde u \in PSH(U)$ such that ${\tilde u}|_V=u.$
\end{proposition}
\begin{proof}
As was noted in the proof of Theorem 2.3 in [Wik], the above statement is buried in the seminal work [FN]. We offer, however, for the convenience of the reader, a full proof of this important fact.
According to a classical result of Siu in [Siu], the variety $V$ admits an open Stein \nhd\ in $\mathbb C^n.$ Thus, by shrinking $D$ if necessary, we may assume $D$ is Stein.
In the Stein open set $D \times \mathbb C$, consider the Hartogs variety
$$\tilde V:=\{(z, w) \in V \times \mathbb C: \log \vert w\vert+u(z)<0\}.$$
Since $u \in PSH(V),$ the variety $\tilde V$ is also Stein. Applying Siu's theorem again, we obtain an open Stein \nhd\
$D' \subset D \times \mathbb C$ of $\tilde V.$ The point is to modify $D'$ to be a Hartogs domain. For this, we proceed as in [FN], p. 64. Let $D''$ be the interior of the set
$$\bigcap\limits_{t \in \C, \vert t\vert \ge 1} \{(z,w) \in D \times \mathbb C: (z,w/t) \in D'\}.$$
Then $D'',$ being the interior of the intersection of a family of Stein open sets in $\C^{n+1}$ is also Stein.
Further, we also have $D'' \cap (V \times \C)=\tilde V.$
Let $\tilde D:= \{(z,w): (z,0) \in D''\}$. Then $\tilde D$ is a Stein open set in $\C^{n+1}$ and
$\tilde D \cap (V \times \C)=\tilde V.$ Moreover, $\tilde D$ enjoys the following key property
$$(z,w) \in \tilde D, \vert w'\vert \le \vert w\vert \Rightarrow (z,w) \in \tilde D.$$
Hence $\tilde D$ is indeed a Hartogs domain in $\C^{n+1}$ and therefore may be represented by
$$\tilde D=\{(z,w) \in U \times \C: \log \vert w\vert +\tilde u(z)<0\},$$
where $U$ is an open subset of $D, V \subset U$ and $\tilde u$ is upper-semicontinuous on $U$. Since $\tilde D$ is Stein we infer that $\tilde u \in \PSH(U)$. Finally, since $\tilde D \cap (V \times \C)=\tilde V,$
by the definition of $\tilde V$ we get that $\tilde u|_V=u.$ This completes our proof.
\end{proof}
\begin{remark}
In the case where $V$ is a Stein {\it smooth} complex variety of $D$ then it is possible to choose a {\it common}
open Stein \nhd\ $U$ of $V$ such that every $u \in PSH(V)$ extends to a \psh\ function on $U.$
To see this, as before, we first choose a Stein \nhd\ $U \subset D$ of $V$. Then by the classical Bishop-Remmert-Narasimhan embedding theorem we may regard $V \subset U$ as closed submanifolds of $\C^{2n+1}.$
Then by an extension result of Sadullaev (see Theorem 3.2 in [BL]) we may extend every $u \in PSH(V)$ to an element
$\tilde u \in PSH(\C^{2n+1})$.
It follows that $\tilde u|_{U}$ is the desired extension of $u.$
It would be of interest to know if the extended function $\tilde u$ can be chosen to be {\it negative} if the initial function $u$ is so.
\end{remark}
\section{The class $\mathcal E_0 (V)$}
The aim of this section is to investigate in details the class $\mathcal E_0 (V) \subset PSH^{-} (V)$ whose elements  serve as test functions in convergence problems of the complex Monge-Amp\`ere operator.

We start with the following sufficient condition for hyperconvexity of a Stein complex variety.
\begin{proposition} \label{hypnhd}
Assume that the variety $V$ is Stein and
that for every $\xi \in \partial V$ there exists $u_{\xi} \in \PSH^{-} (D)$ satisfying
$\lim\limits_{z \to \xi} u_{\xi}(z)=0.$
Then there is a bounded hyperconvex domain $\tilde D$ in $D$ such that $V \subset \tilde D.$ In particular, $V$
is hyperconvex variety having a negative $\mathcal C^\infty-$smooth strictly \psh\ exhaustion function.
\end{proposition}
\noindent
We do not know if the above result is true under the weaker assumption that
$u_\xi \in \PSH^{-} (V)$ for each $\xi \in \partial V.$
If this was the case then we would find a {\it continuous} negative \psh\ function on every hyperconvex variety.
\begin{proof}
Since $V$ is Stein, by the main theorem in [Siu] we may find a Stein variety $D' \subset D$ such that $D' \supset V.$ Let $\{K_j\}$ be a sequence of compact subsets of $V$ such that $K_j \uparrow V.$
For each $j$ we let
$$U_j:= \big \{z \in D': \text{dist}\ (z, K_j)<\min \{\fr1{j}, \fr1{2}\text{dist}\ (K_j, \partial D')\}\big\}.$$
Then $D'':= \bigcup\limits_{j} U_j$ is an open subset of $D'$ that contains $V$ and satisfies
$\partial D'' \cap \partial D' =\partial V.$
Applying Siu's theorem again we can find a Stein variety $D^*$ in $D''$ such that $D^* \supset V.$
Observe that we also have $\partial D^* \cap \partial D'=\partial V.$
Since $D'$ is Stein, $V$ is the common zero set of a {\it finite} number of holomorphic functions
$f_1, \cdots, f_l$ on $D'.$ Let $\va$ be a continuous \psh\ exhaustion function for $D^*$.
Following an idea in the proof of Theorem 2.3 in [Wik2], we set
$\psi:=\va+\log (\vert f_1\vert+\cdots+\vert f_l\vert).$ It is then easy to check that
$$\lim\limits_{z \to \xi} \psi(z)=+\infty, \ \forall \xi \in (\partial D^*) \setminus \partial V.$$
Now we define $\tilde D:=\{z \in D^*: \psi (z)<0\}.$ It follows that $V \subset \tilde D$ and
$\partial \tilde D \cap \partial D^*=\partial V.$ Therefore
\begin{equation} \label{barrier}
\lim\limits_{z \to \xi} \psi(z)=0, \ \forall \xi \in (\partial \tilde D) \setminus \partial V.
\end{equation}
Finally, we let $K \subset \tilde D$ be a closed ball. Define
$$u_{K,\tilde D} (z):=\sup\{u \in PSH^{-} (\tilde D): u|_{K} \le -1\}.$$
It follows that $\tilde u:=u^*_{K,\tilde D} \in PSH(\tilde D), \tilde u \le 0$ on $\tilde D$ and $\tilde u=-1$ on the interior of $K$.
Hence $\tilde u \in PSH^{-}(\tilde D)$, by the maximum principle. It is also clear that
$\sup\limits_{K} u_\xi>-\infty$ for every $\xi \in \partial V.$
Thus, by the assumption we obtain
$$0 \ge \ \varlimsup\limits_{z \to \xi} \tilde u(z) \ge \varliminf\limits_{z \to \xi} \tilde u(z) \ge \lim\limits_{z \to \xi} \fr{1}{-\sup\limits_{K} u_\xi}u_\xi (z)=0, \ \forall \xi \in \partial V.$$
Hence $\lim\limits_{z \to \xi} \tilde u (z)=0$ for every $\xi \in \partial V.$
On the other hand, (\ref{barrier}) implies that
$$0 \ge \varlimsup\limits_{z\to \xi} \tilde u(z) \ge \varliminf\limits_{z \to \xi} \tilde u (z)\ge \lim_{z \to \xi} \psi(z)=0 \
\forall \xi \in (\partial \tilde D) \setminus \partial V.$$
Hence $\tilde u$ is a negative \psh\ exhaustion function for $\tilde D$.
So $\tilde D$ is the desired hyperconvex \nhd\ of $V$. For the last statement, it suffices to take
a negative  $\mathcal C^\infty-$smooth strictly \psh\ exhaustion function $\theta$ of $\tilde D$
(see Theorem 1.6 in [B\l]) and consider the restriction of $\theta$ on $V.$
\end{proof}
\begin{remark} \label{exhaustion1}
If $V$ admits a negative continuous exhaustion function $\rho$ then it also has a $\mathcal C^\infty-$smooth negative strictly \psh\ exhaustion function. This has been done in the case of open domains in [B\l].
The proof for complex varieties $V$ is similar.
For the convenience of the reader, we will briefly outline.
Choose $M>0$ so large such that $\va(z):=\vert z\vert^2-M<0$ on $V.$ Set $\tilde \rho:=-2\sqrt{-\va\rho}.$
By a direct computation as in [B\l], p. 728, we can check that $\tilde \rho$ is a negative continuous {\it strictly}
\psh\ exhaustion function
for $V.$ Now we let $f(z):= \min \{-\rho(z), \text{dist}\ (z, \partial V)\}$. Then $f$ is a positive continuous function on $V$ that tends to $0$ at $\partial V.$ By a classical approximation theorem of Richberg, we can find a
$\CC^\infty-$smooth strictly \psh\ function $\hat \rho$ on $V$ such that $\tilde \rho<\hat \rho <\tilde \rho+f$ on $V$.
In particular, $\hat \rho \in PSH^{-} (V)$ is an exhaustion function for $V.$
\end{remark}
\noindent
From now on, unless otherwise stated, we always assume that $V$ is {\it hyperconvex}, i.e., $V$ admits a negative \psh\ exhaustion function $\rho$, that is
\begin{equation}
\label{exhau}
\rho \in PSH^{-}(V) \cap L^\infty(V), \lim_{z \to \partial V} \rho(z)=0.
\end{equation}
It is not clear to us if the exhaustion function $\rho$ can be chosen to be {\it continuous} on $V.$
See Proposition \ref{hypnhd} above for a partial result.
We now present an approximation result which is of interest in its own right.
\begin{proposition}\label{pro1}
Assume that the function $\rho$ in (\ref{exhau} is continuous on $V.$
Then for every $u \in  PSH^{-} (V)$ there  exists a sequence  $u_j \in PSH^-(V) \cap \CC (\ov{V})$ such that
$u_j|_{\partial  V}= 0$ and  $u_j \downarrow  u$ on $V$.
\end{proposition}
A weaker version of the above result was established, using the same scheme as in the case of domains in $\C^n$ (cf. Theorem 2.1 in [Ce1]), was established in Theorem 2.3 of [Wik1] where the ambient domain $D$ is assumed to be hyperconvex.
Notice also that if the continuity requirement on $u_j$ is dropped then we may simply take $u_j:= \max \{u,j\rho\}.$

The proof of Proposition \ref{pro1} requires some elementary facts about behavior of sequences of upper semicontinuous and lower semicontinuous functions on subsets of $\mathbb C^n.$
\begin{lemma}\label{lm1}
Let $\{ f_j\}_{j \ge 1}$ is a decreasing sequence of upper semicontinuous functions defined on a compact $K \subset \C^n$ and $g$ be a lower semicontinuous continuous  function  on $K$ such that
$$\lim_{j \to \infty} f_j (x) \le g(x),  \forall  x \in K.$$ Then for every
$\varepsilon  > 0$ there exists  $j_0$  such that  if $j \ge j_0$ then
$$f_j (x)<g(x)+\varepsilon,\ \forall  x \in  K.$$
\end{lemma}
\begin{proof}
For $j \ge 1$, we let  $K_j: =\{x \in K: f_j (x)-g(x) \ge \varepsilon \}.$
By the assumptions, we infer that $\{K_j\}_{j \ge 1}$ is a decreasing sequence of compact sets such that
$\bigcap_{j \ge 1} K_j =\emptyset.$
Thus we can find an index $j_0 \ge 1$ such that $K_j=\emptyset$ for $j \ge j_0.$
This proves our lemma.
\end{proof}
\begin{lemma}\label{lm2}
Let $X$  be a subset  of  $\C^n$  and  $\{\va_j\}_{j \ge 1}$  be a  sequence  of  lower semicontinuous  functions  on  $X$  that  increases  to  a lower  semicontinuous  function  $\va$ on $X$.
Then  for every  sequence  $\{a_j\}_{j \ge 1} \subset  X$  with  $a_j  \to  a \in X$ we have
$$\va (a) \le \varliminf\limits_{j\to \infty} \va_j (a_j).$$
\end{lemma}
\begin{proof} For $j \ge k$ we have  $\va_k (a_j) \le \va_j (a_j)$. By letting $j \to \infty$ and using lower semicontinuity of $\va_k$ at $a$ we obtain
$$\va_k (a) \le \varliminf\limits_{j\to \infty} \va_k (a_j)\le\varliminf\limits_{j\to \infty} \va_j (a_j).$$
The desired conclusion follows by letting $k \to \infty$ in the right hand side.
\end{proof}
\begin{lemma}\label{lm3}  Assume that $ \{ f_j\}_{j\ge 1} $ is a sequence of upper- semicontinuous functions defined on a compact $K$ which decrease to an upper-semicontinuous  function $ f$. Then
$$
\lim\limits_{j\to\infty} (\max\limits_K f_j ) = \max\limits_K f.$$
\end{lemma}
\begin{proof}
We have $\max_K f_j \downarrow a \ge \max_ K f$. So it is enough to prove the reverse inequality. Assume that
$\max\limits_K<b<a$ for some $ b$. Then  the sets $A_j:= \{ x \in K : f_j (x) \ge b\}$ is nonempty for all  $j$. Moreover,  form  the upper-semicontinuity of the  $f_j$ it follows that the  sequence $A_j$ is compact  and decreasing, hence  $\bigcap\limits_{j\ge 1} A_j \ne \emptyset $. It implies that there exists  $z \in K$ such that  $f(z) \ge b$. Hence $ \max_K f \ge f(z) \ge b$, this is a contradiction.
\end{proof}
\begin{proof} (Proposition \ref{pro1}) By the assumption, there exists
$\rho \in PSH^{-} (V) \cap \CC (\ov V), \rho|_{\partial V}=0.$
For each  $j\ge 1$,  we put
$V_j:=\{ z \in V: \rho(z)<-\frac1{2j^2}\}$. Then  $V_j \subset \subset V$ and  $V_j \uparrow  V$.
Since $V$ is  Stein,  Fornaess-Narasimhan's approximation theorem yields a sequence of strictly \psh\ functions
$\va_j \in  PSH(V) \cap \CC(V)$  such that $\va_j  \downarrow  u$  on  $V$.
Then because   $  u< 0$ on  $  \ov{V}_j$, by Lemma \ref{lm1} we can choose $l(j)>j$ such that
$\va_{l(j)}<\frac1{2j}$ on  $ \ov {V_j}$.
 We  define
$$\tilde{u}_j(z):=
\begin{cases}
\max  \big\{\va_{l(j)}(z)-\frac1j,  j \rho(z)\big\} &  z \in V_j\\
j\rho(z)& z \in \ov{V} \setminus  V_{j}.
\end{cases}$$
Then on $V \cap \partial V_j $ we have $\va_{l(j)}-\frac1j=(\va_{l(j)} -\frac1{2j}) -\frac1{2j} \le - \frac1{2j} = j\rho$,  thus  $\tilde{u}_j \in  PSH^-(V) \cap  C(\ov V)$  and  it is easy to see that     $\tilde u_j \to u$ pointwise on $ V$ as $j \to \infty$ and   $\tilde{u}_j|_{\partial V} = 0$ .
Fix $j\ge  1$, for  each  $p \ge  j$ we set
$$h_{p,j}:= \sup\limits_{ j\le m\le p} \tilde{u}_m, u_j:= \sup\limits_{j \le m} \tilde{u}_m.$$
It is then clear that $h_{p,j} \in PSH (V) \cap  \CC(\ov V)$ and  $h_{p,j} \uparrow  u_j$ on $\ov{V}$.
Moreover, $u_j \ge \lim\limits_{m \to \infty} \tilde{u}_m=u$ on $V.$
Now we claim that $h_{p,j}$ is uniformly  convergent  to $u_j$  on every compact  subset of  $V$ as  $p \to  \infty$.  Take a compact  $K \subset  V$. If the uniform convergence fails on $K$, then there exist
$\ep > 0$ and  a  sequence   $\{z_p \}_{p\ge 1} \subset K$  satisfying
\begin{equation}\label{eqa}
 h_{p,j}(z_p)+\ep <u_j (z_p),\ \forall  p \ge  j.
 \end{equation}
 By the definition of $u_j$ and since  (\ref{eqa}), there exists $m(p) > p$  such that
\begin{equation}\label{eqb}
 h_{p,j}(z_p)+\frac\ep2<  \tilde{u}_{m(p)}( z_p) ,\ \forall  p \ge  j.
 \end{equation}
After switching  to a subsequence if necessary, we  can assume that $z_p \to z^* \in K.$
By the definition of $\tilde u_{m(p)}$ and taking into account the fact that $\rho<0$ on $V$ we obtain
\begin{equation}\label{eqc}
\tilde{u}_{m(p)}= \va_{l(m(p))}-\frac1{m(p)} \  \text{on}\ K,
\end{equation}
for all $p \ge p_0$ large enough.
Combining (\ref{eqb}), (\ref{eqc}) with  the fact that  $\{\va_p\}_{p\ge1}$ is decreasing we  get
$$h_{p,j}(z_p)+\frac\ep2 \le  \va_{l(m(p))}(z_p)-\frac1{m(p)} \le  \va_{p}(z_p) - \frac1{m(p)}.$$
Thus, by letting $p \to \infty$  and taking limsup in both sides we get
$$u_j(z^*) + \frac\ep2\le  \varlimsup\limits_{p\to\infty} \va_p(z_p) \le   u (z^*),$$
where  the last inequality  follows from  Lemma \ref{lm2}.
This is a contradiction to the fact that  $u_j(z) \ge  u(z)$ on  $V$.
Thus, our claim is fully proved.
Since $h_{p,j} \in PSH^-(V) \cap  \CC(\ov V)$, it follows that $u_j \in PSH^{-}(V) \cap \CC(\ov V)$ with $u_j|_{\partial V}=0.$
Finally, we notice that
$$\lim_{j \to \infty} u_j= \varlimsup\limits_{m\to\infty} \tilde u_m=u.$$
The proof is complete.
\end{proof}
In order to obtain a full generalization of Theorem 2.1 in [Ce2], we impose the following additional conditions on $V.$
\begin{definition} \label{singular}
The variety $V$ is said to be mildly singular if it is
contained in a connected complex variety $V'$ of pure dimension $k$
in some domain $D', D \subset D'$
such that $\ov{V_{red}} \cap \partial V \subset V',$ and $\rho$ in (\ref{exhau}) extends to a bounded \psh\ function on $V'$. Moreover, $V$ is called very mildly singular if $V_{\red}$ is a finite set.
\end{definition}
\begin{remark}
(a) If $V$ is very mildly singular then it is mildly singular since in that case $\ov{V_{red}} \cap \partial V=\emptyset$ so we may take $V'=V$.

\noindent
(b) Let $V$ be a Stein variety in a domain $D \subset \mathbb C^n$ and $\rho \in \PSH(V)$ be a \psh\ exhaustion
function for $V$. Then for each $a>0, V \cap D_a$ is a mildly singular (hyperconvex) variety in $D_a,$
where $D_a$ is the connected component of $\{\rho<a\}$.

\noindent
(c) The Whitney umbrella $V^1$ defined in the introduction is obviously mildly singular and hyperconvex. On the other hand, the variety $V^2$ is not so, since there are infinitely many branches of $V^2$ clustered at some boundary point of $\partial V^1$. Notice that $V^2$ admits a {\it continuous} negative \psh\ exhaustion function which is the restriction on $V^2$ of such an exhaustion function of the hyperconvex domain $\mathbb B.$

\end{remark}
\noindent
The above notions will provides us with a class of varieties $V$ on which $\mathcal E_0 (V) \ne\emptyset.$
\begin{proposition}\label{e0}
Let $V$ be a mildly singular hyperconvex variety and $E$ be an open subset of $V$ (not necessarily connected) such that
$\ov{V_{\red}} \cap V \subset E$ and $\ov{E} \cap \partial V=\ov{V_{red}} \cap \partial V.$

Set $\tilde E:= E \cup (V' \setminus \ov{V}),$ and define
\begin{equation}\label{extremal}
u(z):= \sup\{v(z): v \in PSH(V'), \sup_{V'} v \le \sup_{V'} \rho, v|_{V}<0, u|_{\tilde E} \le \rho \}, z \in V'.
\end{equation}
Then $u$ has the following properties:

\noindent
(a) $u \in \PSH (V') \cap L^\infty (V'), u|_V<0, \lim\limits_{z \to \partial V} u(z)=0;$

\noindent
(b)$(dd^c u)^k=0$ on $V'\setminus \ov{\tilde E};$

\noindent
(c)$\int\limits_V (dd^c u)^k<\infty.$

\noindent
Consequently, for each open relatively compact subset $U$ of $V$ with $V \cap E=\emptyset$
we can find $\rho' \in \mathcal E_0 (V)$ such that $dd^c \rho'=\omega$ on $U.$
\end{proposition}
We require the following lemma which will also be needed in the last section.
\begin{lemma}\label{local}
Let $\mathbb B$ be a ball in $\C^k$ and $f\ge 0$ be a upper-semicontinuous functions on
$\mathbb B$. Let $u \in PSH(\B) \cap L^\infty(\B)$ be such that
$(dd^c u)^k \ge f\om^k.$
Then there exists $\tilde u \in PSH(\B)\cap L^\infty(\B)$ satisfying
$$(dd^c \tilde u)^k=f\om^k, \tilde u \ge u, \lim_{z \to \xi} \tilde u(z)=u^*(\xi), \forall \xi \in \partial \B.$$
\end{lemma}
\begin{proof}
Let $\{\va_j\} \in \CC (\partial \B)$ and $\{f_j\} \in \CC(\B), f_j \ge 0$ be  sequences that decreases to
$u^*|_{\partial \B}$ and $f$ respectively.
For each $j,$ by Theorem 4.1 in [B\l] (which is rooted in Theorem 8.1 in [BT]) we can find
$u_j \in PSH(\B) \cap \CC(\ov \B)$ such that
$$(dd^c u_j)^k=f_j\om^k, u_j|_{\partial \B}=\va_j.$$
By the comparison principle and the assumption we see that $u_j$ is decreasing on $\B$ and $u_j \ge u$ for every $j.$
Set $\tilde u:=\lim\limits_{j\to \infty} u_j.$ It is then clear that $\tilde u\in PSH(\B)\cap L^\infty(\B)$ and satisfies
$$\tilde u \ge u, {\tilde u}^*|_{\partial \B} \le \lim\limits_{j \to \infty} \va_j=u^*.$$
This implies that $\lim\limits_{z \to \xi} \tilde u(z)=u^*(\xi), \forall \xi \in \partial \B.$
Finally, by the weak$^*-$convergent of Monge-Amp\`ere measures, $(dd^c \tilde u)^k=f\om^k$ on $\B.$
\end{proof}
\begin{proof}
(a) From (\ref{extremal}) we infer that $u \ge \rho$ on $V'$ and $u \le \rho$ on $\tilde E.$
Hence $u=u^*=\rho$ on $\tilde E.$ Since $\lim\limits_{z \to \partial V}\rho (z)=0$ we also have
$$\lim\limits_{z \to \partial V, z \in V} u^*(z)=0.$$
Observe that $V'_{\red}$ is {\it disjoint} from $\tilde E,$ so by Proposition \ref{regular} we get
$u^* \in PSH(V') \cap L^\infty (V')$
and $u^*\big\vert_V\in PSH^{-} (V).$ It follows that $u=u^*$ on $V.$ This implies the desired conclusion.

\noindent
(b) Fix $z_0 \in V'_{\reg} \setminus \ov{\tilde E}$.
Then we can choose a small \nhd\ $W$ of $z_0$ in $V$ such that $W$ is biholomorphic to
some ball in $\mathbb C^k.$ By applying Lemma \ref{local} we can find $\tilde u \in PSH(W)$
such that $(dd^c \tilde u)^k=0, \tilde u \ge u$ on $W$ and $\lim\limits_{z \to \xi} \tilde u(z)=u(\xi)$ for every
$\xi \in \partial W.$
Set $\hat u:=u$ on $V' \setminus W$ and $\hat u:= \tilde u$ on $W.$ By Lemma \ref{gluing} we see that
$\hat u \in PSH(V')$. Further, $\hat u \le u$ on $V'$. This implies that $u=\tilde u$ on $W.$ In particular,
$(dd^c u)^k=0$ on $W$. Since $z_0$ is arbitrary in the complement of $V_{\sig}$
and since this measure does not charge $V_{\sig}$ we conclude that it vanishes off $\ov{\tilde E}.$

\noindent
(c) We note that the {\it restriction} of $(dd^c u)^k$ on $V$ is supported in $V \cap \ov{\tilde E}$ which is {\it relatively compact} in $V'.$
Hence, by the Chern-Levine-Nirenberg inequality and boundedness of $u$ we complete the proof of (c).

Finally, let $M>0$ be so large such that $\psi(z):=\Vert z\Vert^2-M<0$ on $\ov{V}.$ Then for $\la>0$ large enough we have $\la u^*<\psi$ on $U.$ Thus the function
$\rho':=\max \{\la u, \psi\} \in PSH^{-} (V) \cap L^\infty (V)$
and satisfies
$\lim\limits_{z \to \partial V} \rho'(z)=0, \rho'=\psi$ on $U$. Moreover, by the comparison principle (Corollary
\ref{corBe}) we obtain
$$\int\limits_V (dd^c \rho')^k \le \la^k \int\limits_V (dd^c u)^k <\infty.$$
Thus $\rho'$ is the desired function.
\end{proof}
The next result summarizes properties of the relative extremal function $u_K$ where $K$ is a compact subset of $V.$
In particular, this gives us a sufficient condition on $V$ so that
there exists a continuous element in $\mathcal E_0 (V)$ whose Monge-Amp\`ere measure has compact support in $V$.
\begin{lemma}\label{pro2}
Let $K$ be a compact subset of $V$. Then the following statements hold true:

\noindent
(a) $u^*_K \in \PSH (V_{\reg})$ and $(dd^c u^*_K)^k=0$ on $V_{\reg} \setminus K.$

\noindent
(b) $u_K$ is lower semicontinuous on $V$ and satisfies
$\lim\limits_{z \to \partial V} u_{K}(z)=0.$

\noindent
(c) Suppose that $u^*_K \in PSH^{-} (V).$ Then $u^*_K \in \E_0 (V)$ and $(dd^c u^*_K)^k$ is supported on $K.$
Moreover, if $u^*_{K}=-1$ on $K$ then $u^*_{K}$ is continuous on $V.$

\noindent
(d) If $V$ is mildly singular then
$$C(K)=\int\limits_{K \cap V_{\reg}}(dd^c u^*_K)^k.$$
\noindent
(e) If $V$ is very mildly singular, $K$ is non-pluripolar and if $V_{\red} \subset K^0,$ the interior of $K$ relative in $V$ then $u^*_{K} \in \PSH^{-}(V).$ In particular $u^*_{K} \in \PSH^{-}(V)$ if $V$ is locally irreducible.
\end{lemma}
\noindent
Lemma \ref{pro2} (c) was proved in Theorem 1.7 of [Wik1] under the additional assumption that $D$ is hyperconvex.
Note that our method is intrinsic (working directly on $V$) and so different from those in [Wik1].
\begin{proof}
(a) follows from the definition of $u_K$ and the Poisson modification method given in Proposition \ref{e0}(b).

\noindent
(b) First, we let
$$\tilde{u}(z):= \sup \{v(z): v \in \PSH^{-} (V) \cap  \mathcal C(V), v\big|_K \le -1\}.$$
Then $\tilde{u}$  is  lower semicontinuous   on  $V$. Hence it is enough  to show that $u_{K}=\tilde{u}$ on $V.$
It is clear that $\tilde{u} \le  u_{K}$. Fix  $v \in PSH(V), v \le -1$ on  $K$ and  $v \le 0$ on  $V$.
Then by Proposition \ref{pro1} there exists a sequence  $v_j  \in  PSH^-(V) \cap \mathcal C(\ov{V})$ with
$v_j|_{\partial  V} = 0$ such that $v_j \downarrow \tilde v:= \max\{v, -1\}$ on  $V$.
Put  $M_j:= \max\limits_K v_j$. Then by Lemma \ref{lm3}, $M_j \downarrow \max\limits_K  \tilde v   = -1$.
Next we set
$\tilde{v}_j:= v_j-(M_j +1)$. Then $\tilde{v}_j \in \PSH^-(V) \cap  \mathcal C(V), \tilde{v}_j< 0$  and
$$\tilde{v}_j\big|_K=(v_j\big|_K-M_j)-1 \le -1.$$
Hence  $v-(M_j +1) \le \tilde{v}_j \le \tilde{u}.$
By letting  $j\to \infty$ we conclude that $v\le \tilde{u}$ on  $V$. Since $v$ is arbitrary we obtain
$u_{K} \le \tilde{u}$ on $V$.
Finally, we choose  $M>0$ so large that  $M\rho \le -1$ on $K$. Then by the construction of $u_{K}$ we have
$M\rho \le u_{K}$. This implies that $\lim\limits_{z \to \partial V} u_{K}(z) = 0$.

\noindent
(c) Obviously $u^*_K \ge -1$ on $V.$ Further, by (a) we have $\lim\limits_{z \to \partial V} u^*_{K}(z)=0.$
Further, since $(dd^c u^*_K)^k$ does not charge $V_{\sig}$ we have $(dd^c u^*_{K})^k=0$ on
$V \setminus K.$ Thus $u^*_K \in \E_0 (V)$. Finally, assume that $u^*_{K}=-1$ on $K$. Then, since $u^*_K \in \PSH^{-}(V)$ we infer from the definition of $u_{K}$ that $u^*_{K}=u_{K}$ on $V$.
In particular  $u_{K}$ is upper semicontinuous on $V.$ This fact, combined with (a), implies continuity of
$u_{K}$ on $V.$

\noindent
(d) Choose $\rho \in \E_0 (V)$ such that $\rho|_K \le -1.$ By Choquet's lemma, we can find a sequence
$\{u_j\} \subset \PSH^{-} (V), u_j|_K \le -1$ and $u_j \uparrow u, u^*=u^*_K$ on $V.$
Set $u'_j:=\max \{u_j, \rho\}.$ Then $u'_j \uparrow u$ on $V$ as well.
Fix $\ve>0$. Since the sequence $\{u'_j\}$ is uniformly bounded, there exists an open \nhd\ $U_\ve$ of $K$ such that
$$\int\limits_{U_\ve} (dd^c u'_j)^k \le \int\limits_{K} (dd^c u'_j)^k \le C(K)+\ve.$$
Since $(dd^c u'_j)^k$ converges to $(dd^c u^*_K)^k$ weakly on $V_\reg$ we obtain
$$\varliminf\limits_{j \to \infty}\int\limits_{U_\ve \cap V_\reg} (dd^c u'_j)^k \ge
\int\limits_{U_\ve \cap V_\reg} (dd^c u^*_K)^k \ge \int\limits_{K \cap V_\reg} (dd^c u^*_K)^k.$$
Combining the two last estimates and letting $\ve \downarrow 0$ we arrive at
$$\int\limits_{K \cap V_\reg} (dd^c u^*_K)^k \le C(K).$$
For the reverse direction, we will modify slightly the original argument in [BT1] (for the case of open domains in
$\mathbb C^n$). More precisely, fix $\ve>0$ and $u \in \PSH^{-} (V)$ with $u \ge -1.$
For each $j$, we define
$$v_j: =\max \{u'_j, (1-2\ve)u-\ve\}.$$
It is clear that $v_j=u'_j$ outside a {\it fixed} compact subset $K \subset K_\ve \subset V$ whereas
$v_j=(1-2\ve)u-\ve$ on
some \nhd\ of $K$. It follows, using Proposition \ref{partformular1}, that
$$(1-2\ve)^k \int\limits_K (dd^c u)^k=\int\limits_K (dd^c v_j)^k \le \int\limits_{K_\ve} (dd^c v_j)^k
=\int\limits_{K_\ve} (dd^c u'_j)^k.$$
By Lemma \ref{outer}, we may choose an open \nhd\ $O_\ve$ of $V_\sig$ such that $C(O_\ve)<\ve.$
Using again weak convergence of $(dd^c u'_j)^k$ towards $(dd^c u^*_K)^k$ on $V_\reg$ as $j \to \infty$ we obtain
$$(1-2\ve)^k \int\limits_K (dd^c u)^k \le \int\limits_{K_\ve \setminus O_\ve} (dd^c u^*_K)^k+\ve \le
\int\limits_{K \cap V_\reg} (dd^c u^*_K)^k+\ve.$$
By letting $\ve \downarrow 0$ we obtain the desired conclusion.

\noindent
(e) Fix $z_0 \in V \setminus K^0.$ By the assumption
$z_0 \not\in V_{\red}.$ Therefore $u^*_{K}$ is \psh\ on a small \nhd\ of $z_0.$ Observe also that
$u^*_K=u_K=-1$ on $K^0$. Thus $u^*_{K} \in PSH (V)$.
Moreover, if $u^*_K=0$ somewhere on $V$ then by the maximum principle we would have $u^*_K=0$ entirely on $V.$
Hence $u_K$ must vanishes somewhere on $V_\reg.$ Then from the definition of $u_K$ we may construct
$v \in \PSH^{-} (V)$ such that $v|_E=-\infty.$
\end{proof}
\begin{remark} \label{negative}
Assume that $V$ is very mildly singular.
Let $\va$ be a $\mathcal C^\infty-$smooth strictly \psh\ exhaustion function for $V$. Then we can find $a>0$ such that $K:=\{z \in V: \va(z) \le a\}$ is a compact regular subset of $V$ that contains $V_{\red}.$ So by the above result
$u_{K} \in \E_0 (V)$ is a negative {\it continuous} \psh\ exhaustion function for $V$.
\end{remark}
We now combine the above lemma with estimates in Lemma \ref{energy}
to provide upper bounds for Monge-Amp\`ere measures in terms of capacity.
\begin{lemma} \label{Holder2}
Let $V$ be a mildly singular hyperconvex variety and
$K$ be a compact subset of $V.$
Then for $u_1, ..., u_k \in \E_0(V)$ we have
$$\int\limits_V(-u_{K})T\le MC(K)^{\fr1{k+1}}.$$
\noindent
where $T:= dd^cu_1\wedge \cdots\wedge  dd^cu_k$ and
$M:= \Big(\int\limits_V (-u_1)(dd^cu_1)^k\Big)^{\frac1{k+1}} \cdots
(\int\limits_V (-u_k)(dd^cu_k)^k\Big)^{\frac1{k+1}}.$

Further, for every relatively compact Borel subset $E$ of $V$ we have
$\int\limits_E T \le MC(E)^{\frac1{k+1}}.$
\end{lemma}
\begin{proof}
Fix $\rho \in \E_0 (V).$ As in the proof of Lemma \ref{pro2}, we can find a sequence $\{v_j\} \subset \E_0 (V)$
such that $v_j|_K \le -1, v_j \ge \rho, v_j \uparrow u$ and $u^*=u^*_K$ on $V.$
Then $v_j \in \E_0 (V)$ and
$$\sup\limits_{j} \int\limits_V (dd^c v_j)^k \le \int\limits_V (dd^c \rho)^k<\infty.$$
Using Lemma \ref{energy}(c) we get
$$\int\limits_V (-v_j)T\le M\Big(\int_V (-v_j)(dd^c v_j)^k\Big)^{\frac1{k+1}}.$$
For each $\ve>0$, we let $K_\ve$ be the closure of the open set $\{\rho<-\ve\}.$ Then $K_\ve$ is a compact subset of
$V.$
We have
$$\begin{aligned}
\int\limits_V(-v_j)(dd^c v_j)^k &=\int\limits_{K_\ve} (-v_j)(dd^c v_j)^k+\int\limits_{V \setminus K_\ve} (-v_j)(dd^c v_j)^k\\
& \le \int\limits_{K_\ve}(dd^c v_j)^k+\sup\limits_{V \setminus K_\ve}(-\rho)\int\limits_{V \setminus K_\ve} (dd^c v_j)^k\\
& \le \int\limits_{K_\ve}(dd^c v_j)^k+\sup\limits_{V \setminus K_\ve}(-\rho)\int\limits_{V \setminus K_\ve} (dd^c \rho)^k.
\end{aligned}$$
Since $(dd^c v_j)^k$ converges weakly to $(dd^c u^*_K)^k$ on $V_\reg$, from Lemma \ref{outer} we infer
$$\varlimsup\limits_{j \to \infty} \int\limits_{K_\ve} (dd^c v_j)^k \le \int\limits_{K_\ve \cap V_\reg}
(dd^c u^*_K)^k=\int\limits_{K \cap V_\reg} (dd^c u^*_K)^k \le C(K),$$
where the last estimate follows from Lemma \ref{pro2} (d). Thus, using Lebesgue's monotone convergence theorem we get
$$\int\limits_V (-u)T \le M\Big [C(K)+\sup\limits_{V \setminus K_\ve}(-\rho)\int\limits_{V} (dd^c \rho)^k\Big])^{\frac1{k+1}}.$$
Letting $\ve \downarrow 0$, we get
$$\int\limits_V(-u)dd^cu_1\wedge \cdots\wedge dd^cu_k\le MC(K)^{\fr1{k+1}}.$$
By the solution to the second problem of Lelong on the smooth locus $V_{\reg}$ (cf. [BT1]),
we have $u=u_K$ on $V$ outside a pluripolar subset $S$ of $V$. Since $T$ does not charge $S$ we obtain the first assertion of the lemma.
Finally, given a Borel relatively compact subset $E$ of $V$, we let $\{K_j\}$ be an increasing sequence of compact subsets in $E$ such that
${\int\limits_{K_j} T} \uparrow {\int\limits_E T}.$
Hence
$$\int\limits_V(-u_{K_j})T\le MC(K_j)^{\fr1{k+1}} \le MC(E)^{\fr1{k+1}}.$$
Observe that for each $j,$ there exists a pluripolar subset $S_j$ of $V$ such that
$u_{K_j}=-1$ on $K_j \setminus S_j.$
Since $T$ does not charge $S_j$ we obtain
$$\int\limits_{K_j}T=\int\limits_{K_j}(-u_{K_j})T\le MC(E)^{\fr1{k+1}}.$$
By letting $j \uparrow \infty$ and using inner regularity of $T$ we complete
the proof of the lemma.
\end{proof}
Now we come to the main result of this section.
\begin{theorem}\label{approximation}
Suppose that $V$ is a mildly singular hyperconvex variety.
Then the following assertions hold true:

\noindent
(a) $\mathcal{E}_0(V)$ is non empty;

\noindent
(b) For every $u \in PSH^-(V)$ there exists a  sequence $\{u_j\}_{j\ge 1} \subset \mathcal{E}_0(V)$
such that $u_j \downarrow  u$ on  $V;$

\noindent
(c) If $V$ admits a continuous negative \psh\ exhaustion function
then $u_j$ in (b) can be chosen to be continuous on $\ov{V}$.
\end{theorem}
\begin{proof}
(a) is a direct consequence of Proposition \ref{e0}.

\noindent
(b) Let $K$ be a regular compact subset of $V,$ then $\rho':=u_K \in \E_0 (V)$ is a continuous negative \psh\ exhaustion
function for $V$ (see Remark \ref{negative}).
By replacing  the exhaustion function $\rho$ in the proof of Proposition \ref{pro1} by $\rho'$,
we obtain a sequence $u_j \in \PSH^-(V)$ with $u_j|_{\partial  V} = 0$ and
$u_j \downarrow  u$ on $V$ and  $u_j \ge j\rho'$ on  $\ov V$. Further, by the comparison principle
(Corollary \ref{corBe}) we obtain
$$\int\limits_V (dd^cu_j)^k  \le j^k \int\limits_V (dd^c \rho')^k  <\infty,\ \forall  j \ge  1.$$
Thus $u_j  \in \mathcal{E}_0(V)$.

\noindent
(c) follows directly from the proof of Proposition \ref{pro1}.
\end{proof}
\section{The class $\mathcal E(V)$}
The goal of this section is to construct the largest possible class of negative \psh\ functions on a mildly singular complex variety $V$ for which the Monge-Amp\`ere operator is well defined. We start with the following small modification of Lemma in [Ce2] from which we may regard elements in $\mathcal E_0 (V)$ as test functions.
\begin{lemma}\label{E0}
Let $\varphi \in {\mathcal C}^\infty_0 (D)$ and $U$ be a relatively compact open subset of $V$ that contains
the support $K$ of $\varphi|_V.$
Then for every $\la>0$ there exist $u_1,u_2 \in \mathcal{E}_0(V) \cap  \mathcal{C}(V)$ such that:

\noindent
(a)$\varphi= u_1-u_2$ on $V;$

\noindent
(b) $dd^c u_1 \ge \la \omega, dd^c u_2 \ge \la \omega$ on $U.$
\end{lemma}
\begin{proof}
According to Proposition \ref{e0}, we may find $\rho \in \mathcal E_0 (V).$
Choose real constants $m >\la,  a<0<b$ and $M>0$ satisfying the following conditions:

\noindent
(i) $\sup\limits_{z \in K} \{\max\{\varphi (z)+m\vert z\vert^2, m|z|^2\}\}<b$ on $K, a <\inf\limits_D \varphi$;\\
(ii) $(m-\la)\vert z\vert^2+\va \in \PSH(D)$;\\
(iii) $M\rho <a-b$ on $U$.

\noindent
Set
$$u_1(z):= \max \{\varphi(z)+ m|z|^2 -b, M\rho(z)\}, u_2(z):=  \max \{m|z|^2 -b, M\rho(z)\}, z \in  V.$$
Then we have  $u_1, u_2\in  \mathcal{E}_0(V) \cap  \mathcal{C} (V)$. From (i) and (iii), it is not hard to check that
$\varphi = u_1-u_2$ on $V$. This is our assertion (a). Finally, (b) follows directly from
(ii) and the fact that
$$u_1(z)= \varphi(z)+ m\vert z\vert^2-b, u_2(z)=m\vert z\vert^2-b \ \text{on}\ U.$$
\end{proof}
The next result states among other things that $\mathcal E_0 (V)$ is a convex cone.
\begin{lemma}\label{finite}
(a) $\E_0 (V)$ is a convex cone. More precisely, for $u, v \in \mathcal{E}_0(V)$ we have
$$\int\limits_V(dd^c (u+v))^k \le 2^k \big [\int\limits_V (dd^c u)^k+\int\limits_V (dd^c v)^k\big].$$

\noindent
(b) For $u_1, \cdots, u_k \in \mathcal{E}_0(V)$ we have
$$\int\limits_V dd^c u_1 \wedge \cdots \wedge dd^c u_k \le A_k
\big[\int\limits_V(dd^c u_1)^k+\cdots+\int\limits_V (dd^c u_k)^k \big],$$
where $A_k>0$ depends only on $k.$
\end{lemma}
\noindent
Even though this lemma is an easy consequence of the energy estimates given in Lemma \ref{energy}, we offer below a more elementary proof which is based only on the comparison principle.
\begin{proof}
(a) We proceed similarly as in Section 2 of [Ce1]. More precisely,
since the measure $(dd^c(u+v))^k$ is locally finite on $V$, for each $\al \in (0,1)$
there exists  $\alpha' \in (\al,1)$ such that
$\int\limits_{\{u = \alpha'v\}}(dd^c(u+v))^k =0$. Then by Theorem \ref{thmBe}
$$\begin{aligned}
\int\limits_V(dd^c (u+v))^k
&=\int\limits_{\{u <\alpha' v\}}(dd^c (u+v))^k + \int\limits_{\{\alpha' v < u\}}(dd^c (u+v))^k\\
&=\int\limits_{\{\frac{1+\alpha'}{\alpha'}u <u+v\}}(dd^c (u+v))^k + \int\limits_{\{(1+\alpha')v <u+v\}}(dd^c(u+v))^k\\
& \le f(\al')\big [\int\limits_V (dd^c u)^k+\int\limits_V (dd^c v)^k\big],
\end{aligned}$$
where $f(\al'):= \max\{ (1+\al')^k, (1+1/\al')^k\}$. By letting $\al \to 1$ we obtain that $f(\al') \to 2^k.$
This proves the desired estimate.

\noindent
(b) Applying (a) repeatedly we obtain a constant $C_k>0$ depends only on $k$ such that
$$C_k \big [\int_V (dd^c u_1)^k+\cdots+\int_V (dd^c u_k)^k \big] \ge \int_V (dd^c (u_1+\cdots+u_k))^k
\ge k \int_V dd^c u_1 \wedge \cdots \wedge dd^c u_k.$$
The desired conclusion follows.
\end{proof}
We also need the following useful result about a sort of monotonicity for Monge-Amp\`ere measures in
$\mathcal E_0(V).$ In the case of domains in $\mathbb C^n$, this lemma is implicitly contained in [Ce2].
\begin{lemma}\label{monotonicity} (monotonicity lemma)
Let $h \in \PSH^{-}(V)$ and $u^p, v^p, (1 \le p \le k)$ be functions in $\mathcal E_0 (V)$.
Assume that $u^p \le v^p$ on $V$. Then
$$\int_V hdd^c u^1 \wedge \cdots \wedge dd^c u^k
\le  \int_V hdd^c v^1 \wedge \cdots \wedge dd^c v^k.$$
\end{lemma}
\begin{proof}
First we prove the lemma under the assumption that $h \in \mathcal E_0(V).$
Then, we may apply Corollary \ref{coropartformular1}  repeatedly to obtain
$$\begin{aligned}
\int_V h dd^c u^1\wedge dd^c u^2 \cdots \wedge dd^c u^k
& = \int_V  u^1 dd^c h\wedge dd^c u^2 \wedge \cdots \wedge dd^c u^k\\
&\le \int_V  v^1 dd^ch\wedge dd^c u^2 \cdots \wedge dd^c u^k\\
& = \int_V  u^2 dd^c v^1\wedge dd^c h \wedge \cdots \wedge dd^cu^k\\
&\le \int_V  v^2 dd^c v^1\wedge dd^ch... \wedge dd^cu^k\\
& = \int_V h dd^c v^1\wedge dd^cv^2\wedge... \wedge dd^c u^k\\
& \le \dots\\
&\le  \int_V  h dd^c v^1\wedge... \wedge dd^cv^k.
\end{aligned}
$$
This is our desired estimate. For general $h,$ according to Theorem \ref{approximation}, we
may choose a sequence $h_j \in \mathcal E_0 (V)$ such that $h_j \downarrow h$ on $V.$ By the above reasoning we have
$$\int_V h_jdd^c u^1 \wedge \cdots \wedge dd^c u^k
\le  \int_V h_jdd^c v^1 \wedge \cdots \wedge dd^c v^k, \ \forall j \ge 1.$$
The proof is now completed by letting $j \to \infty$ and applying Lebesgue monotone convergence theorem.
\end{proof}
We are now ready to deliver a class of negative \psh\ functions on $V$ on which the complex Monge-amp\`ere operator may be well defined. This definition is again modeled on Definition 4.1 in [Ce2].
\begin{definition} \label{defdieulong}
Let $V$ be a mildly singular hyperconvex complex variety and $u \in PSH^{-}(V).$ We say that $u \in \E (V)$ if for every $z_0 \in V$ there exists a relatively compact open \nhd\ $U$ of $z_0$ in $V,$ a decreasing sequence $\{h_j\} \subset \E_0 (V)$ such that $h_j \downarrow u$ on $U$ and $\sup\limits_{j}\int\limits_V (dd^c h_j)^k <\infty.$

\noindent
Further, if the above open set $U$ can be chosen to equal $V$ then we write $u \in \F(V)$.
\end{definition}
\noindent
The theorem below is the key to our definition of the Monge-Amp\`ere operator on $\mathcal E(V).$
Needless to say, it is strongly inspired by Theorem 4.2 in [Ce2]. Note that, even in the case of domains in $\mathbb C^n,$ our proof is slightly from [Ce2] and yields a bit stronger statement.
\begin{theorem}\label{maintheorem1}
Let $V$ be a mildly singular hyperconvex variety.
Let $u^1,... u^m \in \mathcal{E} (V) (1 \le m \le k)$ and
$\mathcal {E}_0(V) \ni \{g^p_j\} \downarrow  u^p$ on $V$ as $j \to \infty, 1\le  p\le  m$. Then  the sequence of currents
$dd^c g^1_j\wedge... \wedge  dd^c g^m_j$ is  weak$^*$-convergent on $V$ and  the limit does not  depend  on the  particular sequences $\{g^p_j\}$.
\end{theorem}
The proof given below is
\begin{proof}
Fix $z_0 \in V,$ it is enough to show  $dd^c g^1_j \wedge \cdots \wedge dd^c g^m_j$ is weak $^*-$ convergent on some small \nhd\ $U$ of $z_0$ in $V$ and that the convergent does not depend on the sequences
$\{g^p_j\}, 1 \le p \le m.$
For this, by the definition of $\mathcal E(V)$, we may choose a relatively compact open subset $U$ of $V$ and
decreasing sequences $\{h^p_j\}, (1 \le p \le m)$
in $\mathcal E_0 (V)$ such that $\{h^1_j\} \downarrow u^1, \cdots, \{h^m_j\} \downarrow u^m$ on $U$ and
$$\sup_{1 \le p \le m} \sup_{j \ge 1} \int_V (dd^c h^p_j)^k<\infty.$$
Fix a smooth form $\theta$ of bidegree $(k-m,k-m)$ on $D$ such that $\theta|_V$ has compact support in $U.$
We must show the following assertions:

\noindent
(a)$\exists \underset{j \to \infty}{\lim}\int\limits_V \theta \wedge dd^c g^1_j \wedge \cdots \wedge dd^c g^m_j;$

\noindent
(b) The limit does not depend on the sequences $\{g^p_j\} (1 \le p \le m)$.

Since $\theta$ can be expressed as a linear combination of strongly positive $(k-m,k-m)$ forms on $V$, we may assume
$\theta$ is such a strongly positive form. Moreover, after a linear change of coordinates (in $\mathbb C^n$) we may achieve that $\theta=\va \om^{k-m}$ where $\va$ is a smooth function such that $\va|_V$ is
supported in $U.$
Now for (a), by Proposition \ref{e0}, we may choose $\rho \in \E_0 (V)$ such that $dd^c \rho=\om$ on $U$.
Next, we define for $j,l \ge 1$
$$g^p_{j,l}:=\max\{g^p_j, \tilde h^p_l\}, 1 \le p \le m.$$
It is clear that $g^p_{j,l} \downarrow g^p_j$ on $U$ as $l \to \infty$ and $j$ is fixed. By Bedford-Taylor monotone convergence theorem
$dd^c g^1_{j,l} \wedge \cdots \wedge dd^c g^m_{j,l} \wedge (dd^c \rho)^{k-m}$ is weak $^*-$ convergent to
$dd^c g^1_j \wedge \cdots \wedge dd^c g^m_j \wedge (dd^c \rho)^{k-m}$ on $U.$ So we can find $l(j)$ such that
$$\begin{aligned}
&\big \vert \int_V \va dd^c g^1_{j,l(j)} \wedge \cdots \wedge dd^c g^m_{j,l(j)} \wedge \om^{k-m}-
\int_V \va dd^c g^1_j \wedge \cdots \wedge dd^c g^m_j \wedge \om^{k-m}\big \vert\\
&= \int_V \va dd^c g^1_{j,l(j)} \wedge \cdots \wedge dd^c g^m_{j,l(j)} \wedge (dd^c \rho)^{k-m}-
\int_V \va dd^c g^1_j \wedge \cdots \wedge dd^c g^m_j \wedge (dd^c\rho)^{k-m}\big \vert <\frac1{j}.
\end{aligned}$$
Of course we may arrange so that $l(j)$ is increasing.
Set $$\tilde g^p_j:=g^p_{j,l(j)}, 1 \le p \le m, j \ge 1.$$
By the comparison principle we have $\tilde g^p_j \in \mathcal E_0 (V),$ for $1 \le p \le m$ and
$$\sup_{1 \le p \le m} \sup_{j \ge 1} \int_V (dd^c \tilde g^p_j)^k \le \sup_{1 \le p \le m} \sup_{j \ge 1} \int_V(dd^c \tilde h^p_j)^k<\infty.$$
Moreover, $\tilde g^p_j$ is decreasing entirely on $V$ and
$\tilde g^p_j \downarrow u^p$ on $\om$ as $j \to \infty.$
Next, by applying Lemma  \ref{finite} (b) to
$u^1:=\tilde g^1_j, \cdots, u^m:=\tilde g^m_j, u^{m+1}=\cdots=u^k=\rho$ we obtain
$$\sup_{j \ge 1} \int_V dd^c \tilde g^1_j \wedge \cdots \wedge dd^c \tilde g^m_j \wedge (dd^c \rho)^{k-m}<\infty.$$
Now, by Lemma \ref{E0} we may write $\va=h_1-h_2$ with $h_1, h_2 \in \mathcal E_0 (V).$
By Lemma \ref{monotonicity} we deduce that
the sequence
$\int_V h_1 dd^c \tilde g^1_j \wedge \cdots \wedge dd^c \tilde g^m_j \wedge (dd^c \rho)^{k-m}$ is decreasing and bounded from below by
$$(\inf_V h_1) \sup_{j \ge 1} \int_V dd^c \tilde g^1_j \wedge \cdots \wedge dd^c \tilde g^m_j
\wedge (dd^c \rho)^{k-m}>-\infty.$$
Thus it converges to some (finite) limit. By the same argument
$\int_\Om h_2 dd^c \tilde g^1_j \wedge \cdots \wedge dd^c \tilde g^m_j \wedge (dd^c \rho)^{k-m}$ is convergent too.
Therefore there exists a limit
$$\exists \lim_{j \to \infty} \int_V \va dd^c \tilde g^1_j \wedge \cdots \wedge dd^c \tilde g^m_j \wedge
(dd^c \rho)^{k-m} =\al \in \mathbb R.$$
It follows that
$$\lim_{j \to \infty} \int_V \va dd^c g^1_j \wedge \cdots \wedge dd^c g^m_j \wedge \om^{k-m}
=\lim_{j \to \infty} \int_V \va dd^c g^1_j \wedge \cdots \wedge dd^c g^m_j \wedge (dd^c \rho)^{k-m}=\al.$$
This proves (a).
For the assertion (b), we let $\{v^p_j\} (1 \le p \le m)$ be another sequences in $\mathcal E_0 (V)$ such that $v^p_j \downarrow u^p$
on $V$ for each $1 \le p \le m.$ Next we define $\tilde v^p_j$ analogously as $\tilde g^p_j.$
It suffices to show
$$\lim_{j \to \infty} \int_V \va dd^c \tilde g^1_j \wedge \cdots \wedge dd^c \tilde g^m_j \wedge (dd^c \rho)^{k-m}
=\lim_{j \to \infty} \int_V \va dd^c \tilde v^1_j \wedge \cdots \wedge dd^c \tilde v^m_j \wedge (dd^c \rho)^{k-m}.$$
Hence it is enough to prove the following two assertions
\begin{equation}\label{eq31}\lim_{j \to \infty} \int_V h_1 dd^c \tilde g^1_j \wedge \cdots \wedge dd^c \tilde g^m_j
\wedge (dd^c \rho)^{k-m}=\lim_{j \to \infty}
\int_V h_1 dd^c \tilde v^1_j \wedge \cdots \wedge dd^c \tilde v^m_j \wedge (dd^c \rho)^{k-m},
\end{equation}
\begin{equation}\label{eq32}
\lim_{j \to \infty} \int_V h_2 dd^c \tilde g^1_j \wedge \cdots \wedge dd^c \tilde g^m_j
\wedge (dd^c \rho)^{k-m}=\lim_{j \to \infty}
\int_V h_2 dd^c \tilde v^1_j \wedge \cdots \wedge dd^c \tilde v^m_j \wedge (dd^c \rho)^{k-m}.
\end{equation}
Obviously, it suffices to show (\ref{eq31}) since the proof of (\ref{eq32}) is completely similar.
To this end, we set $\hat g^p_{j,l}:=\max \{\tilde g^p_j, \tilde v^p_l\}$ for $1 \le p \le m$.
Since $\{\tilde v^p_j\}$ and $\{\tilde g^p_j\}$ decrease to the same limit as $j \to \infty$ for each $p$ we infer that
$\hat g^p_{j ,l} \downarrow \tilde g^p_j$ on $V$ as $l \to \infty.$
It follows from Bedford-Taylor's monotone convergence theorem that
$h_1dd^c \hat g^1_{j,l} \wedge \cdots \wedge dd^c \hat g^m_{j, l} \wedge (dd^c \rho)^{k-m}$ is weak $^*-$ convergent to
$h_1dd^c \tilde g^1_j \wedge \cdots \wedge dd^c \tilde g^m_j \wedge (dd^c \rho)^{k-m}$ on $V.$
Fix $\ve>0$.
For each $j$ we can find $l'(j) \ge j$ so large such that
$$\begin{aligned}
\int_V h_1 dd^c \tilde v^1_{l'(j)} \wedge \cdots \wedge dd^c \tilde v^m_{l'(j)} \wedge (dd^c \rho)^{k-m}
&\le \int_V h_1 dd^c \hat g^1_{j,l'(j)} \wedge \cdots \wedge dd^c \hat g^m_{j,l'(j)} \wedge (dd^c \rho)^{k-m}\\
&\le\int_V h_1 dd^c \tilde g^1_j \wedge \cdots \wedge dd^c \tilde g^m_j \wedge (dd^c \rho)^{k-m}+\ve.
\end{aligned}$$
Here the first inequality follows from the monotonicity lemma.
By letting $j \to \infty$ we obtain
$$\lim_{j \to \infty} \int_V h_1 dd^c \tilde v^1_j \wedge \cdots \wedge dd^c \tilde v^m_j  \wedge (dd^c \rho)^{k-m}
\le \lim_{j \to \infty} \int_V h_1 dd^c \tilde g^1_j \wedge \cdots \wedge dd^c \tilde g^m_j \wedge (dd^c \rho)^{k-m}+\ve.$$
Since $\ve>0$ is arbitrary we get
$$\lim_{j \to \infty} \int_V h_1 dd^c \tilde v^1_j \wedge \cdots \wedge dd^c \tilde v^m_j \wedge (dd^c \rho)^{k-m}
\le \lim_{j \to \infty} \int_V h_1 dd^c \tilde g^1_j \wedge \cdots \wedge dd^c \tilde g^m_j \wedge (dd^c \rho)^{k-m}.$$
By exchanging the roles of $\tilde v^p_j$ and $\tilde g^p_j$ we obtain the reverse inequality. This proves (\ref{eq31}) and also the theorem.
\end{proof}
The above result enables us to make the following crucial definition.
\begin{definition} \label{mongeampere}
Let $V$ be a mildly singular hyperconvex variety and $u^1,... u^m \in \mathcal{E} (V) (1 \le m \le k)$.
Then we define $dd^c u^1 \wedge \cdots \wedge dd^c u^m$ to be the limit current given by Theorem \ref{maintheorem1}.
\end{definition}
\begin{definition}\label{Kclass}
Consider a subset $\mathcal{K}$ of $PSH^{-} (V)$ that satisfies following two conditions:

\noindent
(i) If $u \in \mathcal{K}, v \in PSH^{-}(V)$ then  $\max \{u, v\} \in \mathcal{K}$.

\noindent
(ii) If  $u \in \mathcal{K},\va_j \in PSH^{-}(V) \cap L^\infty_{\loc}(V), \va_j \downarrow u$ on  $V$  as
$j \to \infty$ then  $(dd^c \va_j)^k$ is  weak$^*-$convergent on $V$.
\end{definition}
The following result which is again modeled on Theorem 4.5 in [Ce2] shows that $\E$ is the  largest  class  which  (i)  and (ii) in  Definition \ref{Kclass} holds.
\begin{theorem} \label{bigclass}
Let $V$ be a mildly singular hyperconvex variety.
Then the following statements hold true:

\noindent
(a) $\E (V)$ satisfies conditions (i) and (ii) in Definition \ref{Kclass}.

\noindent
(b) Let $\mathcal{K}$ be a sub-class of $\PSH^{-}(V)$
having the properties (i) and (ii) in Definition \ref{Kclass} and $u \in \mathcal{K}.$
Then for every open \nhd\ $U \subset V$ of $\ov{V_\red} \cap V$ satisfying $\ov{U} \cap \partial V=\ov{V_\red} \cap \partial V,$
there exists  sequence $\va_j \in \E_0 (V)$ such that $\va_j \downarrow u$ on $U,$ and for each compact subset
$K$ of $\ov{U} \cap V$ we have
\begin{equation} \label{condition}
\sup\limits_{j \ge 1} \int\limits_{(V \setminus \ov{U}) \cup K} (dd^c \va_j)^k<\infty.
\end{equation}
\end{theorem}
\begin{remark}
If $V$ is very mildly singular then $\ov{V_\red} \cap \partial V=\emptyset.$ So by taking $K=\ov{U}$ in that case we conclude that $\mathcal {K}=\E(V).$
\end{remark}
\begin{proof}
(a) Let $u \in \E (V), v \in PSH^{-}(V)$ and a point $z_0 \in V$.
Since $u \in \E (V)$, we may choose an open \nhd\ $U \subset \subset V$ of $z_0$
and a  decreasing  sequence $u_j \in \E_0(V)$ such that $u_j \downarrow  u$ on  $U$ and
$\sup\limits_j \int\limits_V( dd^c u_j)^k < \infty$.
Put $\va_j:= \max \{u_j, v\}.$ Then $\va_j \in \E_0(V)$ by the comparison principle and
$\va_j \downarrow \max \{u,v\}$ on  $V$.
Moreover, applying again the comparison principle we obtain
$$\sup_j \int\limits_V(dd^c \va_j)^k \le  \sup\limits_j \int\limits_V(dd^c u_j)^k<\infty.$$
Hence $\max\{u, v\}\in \E (V)$ and then $\E (V)$ satisfies (i).
For (ii), let $\va_j \in PSH^{-}(V) \cap L^\infty_{\loc}(V)$ be such that $\va_j \downarrow u$ on  $V$.
By Proposition \ref{E0}, there exists an element $\rho \in \E_0(V)$. Choose a sequence $M_j \uparrow \infty$  such that
$M_j \sup\limits_{U}\rho <\inf\limits_U \va_j.$ Now we set $\tilde{\va}_j:= \max \{\va_j, M_j\rho\}$. Then
$\tilde{\va}_j \in \E_0(V)$,  $\tilde{\va}_j \downarrow u$ on  $V$  and $  \tilde{\va}_j=\va_j$ on $U$.
By Theorem  \ref{maintheorem1} we obtain that $(dd^c \tilde{\va}_j)^k$  is  weak$^*-$convergent on $V$. In particular
$(dd^c \va_j)^k$  is  weak$^*-$convergent on $U$. Since $z_0$ is an arbitrary point in $V$ we conclude that
$\E (V)$ satisfies (ii).

\noindent
(b) Set $u_j: =\max\{u, j\rho\}, j \ge 1.$ Then $\E_0 (V) \ni u_j\downarrow u$ on $V.$
Put
$$\va_j:= \sup \{v \in \PSH^{-} (V): v \le u_j \ \ \text{on}\ U\}.$$
Since $u_j$ belongs to the defining family for $\va_j, \va_j \ge u_j$ on $V,$
and $\va_j \le u_j$ on $U$. Therefore $\va_j=u_j$ on $U.$
On the  other  hand, because  $V$ is locally irreducible at every point of  $V \setminus U$ we infer that
$\va^*_j \in  PSH^{-}(V).$ Hence  $\va_j^*$  also belongs to the constituting family  for $\va_j$.
This implies that $\va_j = \va^*_j \in PSH^{-}(V)$. It follows that $\va_j \in \E_0 (V).$
Furthermore, by the same reasoning as in the proof of
Proposition \ref{e0}(b) we see that
$$\text{supp}\ (dd^c \va_j)^k \subset \overline{U} \cap V.$$
Notice that  $\va_j$ is  decreasing,  $\va_j \ge u$ on  $V$ and  $\va_j \downarrow  u$ on  $U$.
Thus, using (i) we obtain that $\va_j \downarrow \tilde u \in \mathcal K.$
By condition (ii) we have  $(dd^c \va_j)^k$ is weak$^*-$convergent to some measure $\mu$ on $V$.
It is then clear that $\mu$ vanishes off $\ov{U} \cap V$.
In particular $\sup\limits_{j \ge 1} \int_{K} (dd^c \va_j)^k<\infty$ for each compact subset $K$ of $\ov{U} \cap V.$
Combining these facts we see that $\va_j$ satisfies the condition (\ref{condition}).
The proof is thereby completed.
\end{proof}
\begin{remark} We have shown, in part (a) of the theorem, that if
$u \in \E(V),\va_j \in PSH^{-}(V) \cap L^\infty_{\loc}(V), \va_j \downarrow u$ on  $V$  as
$j \to \infty$ then  $(dd^c \va_j)^k$ is  weak$^*-$convergent to $(dd^c u)^k$ on $V.$
\end{remark}
\noindent
Our next main result is an analogue of the monotone convergence theorem of Bedford and Taylor for the class $\E (V)$.
\begin{theorem} \label{maintheorem2}
Assume that $V$ is a mildly singular hyperconvex variety.
Let $\{u_j\} \in \E (V)$ be a sequence that decreases to $u \in \E(V)$.
Then $(dd^c u_j)^k$ converges in the weak$^*-$topology of currents to $(dd^c u)^k$ on $V.$
\end{theorem}
\begin{proof}
Fix $\rho \in \E_0 (V).$ After replacing $\rho$ by $\max\{\rho, u_1\}$ we may assume $\rho \ge u_1$ on $V.$
Let $\theta$ be a smooth function with compact support in $V.$
Now we construct, by induction, a sequence $\{v_j\} \in \E_0 (V)$ such that:

\noindent
(i) $v_0=\rho, \max \{u_j, v_{j+1}\} \le v_j \le \max\{u_j, \rho\}$ on $V;$

\noindent
(ii) $\Big \vert \int\limits_V \theta (dd^c v_{j+1})^k-\int\limits_V \theta (dd^c u_{j+1})^k \Big \vert<\fr1{j+1}.$

For this, it suffices to let $v_{j+1}: =\max \{u_{j+1}, p_{j+1} v_j\}$, where $\{p_j\}$ is an increasing sequence
which is chosen so that (ii) is satisfied. This is possible in view of Theorem \ref{maintheorem1} and the fact that
$\E_0 (V) \ni \max \{u_{j+1}, l v_j\} \downarrow u_{j+1}$ as $l \to \infty$.
By (i) we infer that $v_j \downarrow u$. So applying Theorem \ref{bigclass} we see that $(dd^c v_j)^k$ is
weak$^*-$convergent to $(dd^c u)^k$ on $V.$ Combining this with (ii) we have
$\lim\limits_{j \to \infty}\int\limits_V \theta(dd^c u_j)^k=\int\limits_{V} \theta(dd^c u)^k.$
The proof is thereby completed.
\end{proof}
\begin{remark}
For a hyperconvex domain $D$ in $\C^n,$ Cegrell proved that the complex Monge-Amp\`ere operator is continuous  on
$\F (D)$ with respect to the convergence in {\it capacity}. This result implies our Theorem \ref{maintheorem2} since
monotone convergence in $PSH(D)$ is stronger than convergence in capacity (see Theorem 3.4 in [BT2]) and the fact that each function in $\E (D)$ is {\it locally} the restriction of an element in $\F (D).$
Nevertheless, even in this special case, our proof is much simpler than the one given by Cegrell.
\end{remark}
\noindent
Following Cegrell (cf. [Ce1], [Ce2]), we now introduce some subclasses of $\E(V)$ and $\F(V)$ that will be useful in solving the Dirichlet problem.
\begin{definition}
Let $V$ be a mildly singular hyperconvex variety and $u \in \E(V).$
Then we say $u \in \E_1 (V)$ if there exists a sequence $\{u_j\} \subset \E_0 (V)$ such that $u_j \downarrow u$ on $V$
and $$\sup\limits_{j}\int\limits_V (-u_j)(dd^c u_j)^k <\infty.$$
Finally, $u \in \F_1 (V)$ if the above sequence $\{u_j\}$ satisfies the additional property that
$\sup\limits_{j}\int\limits_V (dd^c u_j)^k <\infty.$
\end{definition}
We now collect below some basic properties of $\E_1 (V)$ and $\F_1 (V)$.
\begin{proposition}\label{classE1}
Let $V$ be a mildly singular hyperconvex variety. Then the following assertions hold true:

\noindent
(a) $\E_1 (V)$ and $\F_1 (V)$ are convex cones.

\noindent
(b) If $u \in \E_1 (V)$ (resp. $\F_1 (V)$) and $v \in PSH^{-} (V), v \ge u$ then
$v \in \E_1 (V)$ (resp. $\F_1 (V)$).

\noindent
(c) If $u \in \E_1 (V)$ then $\int\limits_V (-u)(dd^c u)^k <\infty.$

\noindent
(d) For $u, v \in \E_1 (V)$ we have
$$\int\limits_V (-u)(dd^c v)^k
\le \Big (\int\limits_V (-u)(dd^c u)^k\Big)^{\fr1{k+1}}\Big (\int\limits_V (-v)(dd^c v)^k\Big)^{\fr{k}{k+1}}.$$
\noindent
(e) If $V$ is very mildly singular and
$\{u_j\} \in \E_0 (V)$ satisfies
$\sup\limits_j\int\limits_V (-u_j)(dd^c u_j)^k <\infty.$ Then the sequence of measures $\{(dd^c u_j)^k\}$
has the ACC property. In particular, $(dd^c u)^k$ does not charge pluripolar sets
in $V$ for every $u \in \E_1 (V).$
\end{proposition}
\begin{proof} (a) The convexity of $\E_1(V)$ and $\F_1(V)$ follows from that of $\E_0(V)$
(cf. Lemma \ref{finite} (a)) and the energy estimate Lemma \ref{energy}. The details are very similar to the case of domains in $\C^n$ (cf. Lemma 3.3 in [Ce1]).

\noindent
(b) follows from Lemma \ref{monotonicity} and Lemma \ref{energy} in the same fashion as Lemma 3.4 of [Ce1].

\noindent
(c) We use an idea given in Theorem 3.8 of [Ce1]. Let $u_j \in \E_0 (V)$ be a sequence such that
$u_j \downarrow u$ on $V$ and
$M:=\sup\limits_{j \to \infty} \int\limits_V (-u_j)(dd^c u_j)^k <\infty.$ By Theorem \ref{maintheorem1},
$(dd^c u_l)^k$ is weak$^*-$convergent to $(dd^c u)^k$ on $V.$ Combining this with lower semi-continuity of $-u_j$ we obtain for each $j \ge 1$
$$\begin{aligned}
\int_V (-u_j)(dd^c u)^k & \le \varliminf\limits_{l\to \infty} \int\limits_V (-u_j)(dd^c u_l)^k \\
&\le \Big (\int\limits_V (-u_j)(dd^c u_j)^k\Big)^{\fr1{k+1}} \sup_l \Big (\int\limits_V (-u_l)(dd^c u_l)^k\Big)^{\fr{k}{k+1}}\le M.
\end{aligned}$$
Since $-u_j \uparrow -u$ on $V$, Lebesgue monotone convergence's theorem yields the desired conclusion.

\noindent
(d) We apply Lemma \ref{energy} and Lebesgue monotone convergence's theorem in the same fashion as we did in (c). The details are omitted.

\noindent
(e) is an immediate consequence of Lemma \ref{Holder2}.
\end{proof}
We now present a version of the domination principle for $\E_1 (V).$
\begin{theorem}\label{E1domination}
Assume $V$ is a mildly singular hyperconvex variety. Let $u, v \in \E_1 (V)$ be such that $(dd^c u)^k \le (dd^c v)^k.$
Then we have $u \ge v$ on $V.$
\end{theorem}
\begin{proof}
We first show that $u \le v$ on $V_\reg.$ Assume otherwise, then $u(z_0)<v(z_0)$ for some $z_0 \in V_\reg.$
Set $\psi(z):=\vert z\vert^2-M$ where $M:= \sup\limits_{z \in V} \vert z\vert^2.$
Then we may find $t>0$ so small that $u(z_0)<v(z_0)+t\psi(z_0)$ and $\int\limits_{\{u=v+t \psi\}} (dd^c u)^k=0.$
It follows that
$\int\limits_{\{u<v+t\psi\}} (dd^c \psi)^k>0.$
Let $\E_0 (V) \ni \{u_j\} \downarrow u, \E_0 (V) \ni \{v_l\} \downarrow v$ be such that
$$\sup\limits_j \int\limits_V(-u_j)(dd^c u_j)^k<\infty, \sup\limits_l \int\limits_V(-v_l)(dd^c v_l)^k<\infty.$$
Applying the comparison principle (cf. Theorem \ref{thmBe}) to $u_j, v_l+t\psi$ we get
$$\begin{aligned}
\int\limits_{\{u<v_l+t\psi\}}(dd^c u_j)^k &\ge \int\limits_{\{u_j<v_l+t\psi\}} (dd^c u_j)^k
\ge \int\limits_{\{u_j<v_l+t\psi\}} (dd^c (v_l+t\psi))^k\\
& \ge \int\limits_{\{u_j<v+t\psi\}} (dd^c v_l)^k+t^k \int\limits_{\{u_j<v+t\psi\}} (dd^c \psi)^k.
\end{aligned}$$
By Lemma \ref{classE1} (d), each of the sequences $\{(dd^c u_j)^k\}$ and $\{(dd^c v_l)^k\}$
has the ACC property. Thus, by letting $j\to\infty$ using Lemma \ref{ACC} and taking into account
$\{u_j<v_l+t\psi\} \uparrow \{u<v+t\psi\}$ we obtain
$$\int\limits_{\{u<v_l+t\psi\}} (dd^c u)^k \ge
\int\limits_{\{u<v+t\psi\}} (dd^c v_l)^k+t^k \int\limits_{\{u<v+t\psi\}} (dd^c \psi)^k.$$
Next, we let $l \to \infty$ and use the fact that $\{u<v_l+t\psi\} \downarrow \{u \le v+t\psi\}$ to get
$$\int\limits_{\{u<v+t \psi\}} (dd^c u)^k
=\int\limits_{\{u \le v+t \psi\}} (dd^c u)^k \ge \int\limits_{\{u<v+t \psi\}}(dd^c v)^k+
t^k \int\limits_{\{u<v+t\psi\}} (dd^c\psi)^k.$$
This is impossible since $(dd^c u)^k \le (dd^c v)^k$ by the assumption.
Hence $u \le v$ on $V_\reg.$
Now, we fix $a \in V_{\sig}$, we claim that there exists a one dimensional complex subvariety $\gamma \subset V$ such that $\gamma \cap V_{\sig} =\{a\}$. To see this, we first
make a change of coordinates to find a polydisc $\De$ in $\mathbb C^n$ that contains $a$ and a polydisc $\De'$ in
$\mathbb C^k$ such that the projection map
$\pi: (z_1, \cdots, z_n) \mapsto (z_1, \cdots, z_k)$ expresses $V \cap \De$ as a
branched cover of $\De'=\pi(\De)$ which is branched over a
proper complex subvariety $H$ of $\De'$.
Thus we can find a complex line $l \subset \mathbb C^k$ passing through $\pi(a)$ such that $l \cap H$ is discrete.
Since $\pi(V_{\sig} \cap \De) \subset H$, we have $\gamma \cap V_{\sig}=\{a\}$, where $\gamma:=\pi^{-1} (U)$ and $U \subset l$ is a small \nhd\ of $\pi(a) \in l.$ This proves our claim. Next, we pick an irreducible branch $\gamma' \subset \gamma$ that contains $a.$
Then, by normalization we can find a connected Riemann surface $\gamma^*$ and holomorphic mapping
$f: \gamma^* \to \gamma$ which is surjective.
Set $u':=u \circ f|_{\gamma'}, v':=v \circ f|_{\gamma'}.$ Since $u \ge v$ on $V_{\reg},$ by the choice of $\gamma$, we infer that $u' \ge v'$ on $\gamma^*$ except for the finite set $f^{-1} (a)$. Hence, this inequality holds true entirely on $\gamma^*$ since $u', v'$ are subharmonic there. It follows that $u(a) \ge v(a).$ Hence $u \ge v$ on $V_{\sig}$ as well. The proof is complete.
\end{proof}
\noindent
Using Lemma \ref{energy}, by the same reasoning as in Lemma 3.9 of [Ce1] we have the following result that help to create elements in $\E_1 (V)$ and $\F_1 (V).$
\begin{lemma} \label{ef1}
Let $V$ be a mildly singular hyperconvex variety and $\{h_j\} \in \E_0 (V)$ be a sequence such that
$\lim\limits_{j\to \infty}\int\limits_V h_j(dd^c h_j)^k=0.$ Then the following assertions hold true:

\noindent
(a) There exists a subsequence $\{h_{m_j}\}$ such that $\sum\limits_j h_{m_j} \in \E_1 (V).$

\noindent
(b) If in addition $\lim\limits_{j \to \infty} \int\limits_V (dd^c h_j)^k=0$ then there exists a further subsequence
$\{h_{l_{m_j}}\}$ such that $\sum\limits_j h_{l_{m_j}} \in \F_1 (V).$
\end{lemma}
\noindent
From Lemma \ref{ef1}, it is not hard to derive the following somewhat surprising fact that originally proved by Cegrell in the case of domains in $\C^n$ (cf. Theorem 5.8 in [Ce2]).
\begin{proposition} \label{pluripolar}
Let $V$ be a locally irreducible hyperconvex variety and $E$ be a pluripolar subset in $V.$ Then there exists $u \in \F_1 (V)$ such that $u|_E =-\infty.$
\end{proposition}
\begin{proof}
We may assume that $E$ is a Borel pluripolar subset of $V.$ It follows, using Lemma \ref{outer}, that we may cover $E$
by a collection of countable open sets $\{O_j\}$ such that $O_j$ is relatively compact in $V$ and $C(O_j)<1/j$ for each $j.$ Set $h_j:=u^*_{O_j, V}$. Since $V$ is locally irreducible, we get $h_j \in \E_0 (V), h_j|_{O_j}=-1.$
Moreover $\int\limits_V (dd^c h_j)^k<1/j$ for each $j.$ Now, we apply Lemma \ref{ef1} to get a subsequence
$\{h_{l_j}\}$ such that $u:=\sum\limits_j h_{l_j} \in \F_1 (V).$ Since $u|_E \equiv -\infty$ we conclude the proof.
\end{proof}

\section{The Dirichlet problem}
We start with the following generalization of a classical result due to Bedford and Taylor in [BT1] on solvablility of the complex Monge-Amp\`ere equations for measures which are absolutely continuous with respect to volume forms.
\begin{theorem}\label{BT}
Let $V$ be a locally irreducible hyperconvex variety in a bounded domain $D$ in $\mathbb C^n.$
Let $\va$ be an upper semicontinuous function on $\partial V$ such that there exist
$\psi_1, \psi_2 \in PSH(V) \cap \CC (V)$ satisfying
$$\lim\limits_{z \to \xi} \psi_1(z)=-\lim\limits_{z\to \xi} \psi_2(z)=\va(\xi), \ \forall \xi \in \partial V.$$
Then for every upper-semicontinuous function $f$ on $V$ such that
$$0 \le f\om^k \le (dd^c \psi_1)^k$$
there exists a unique $u \in PSH (V) \cap L^\infty (V)$ satisfying
$$\lim\limits_{z\to \xi}u(z)=\va(\xi) \  \forall \xi \in \partial V, (dd^c u)^k=f\om^k.$$
\end{theorem}
\begin{proof}
The uniqueness of $u$ follows easily from the comparison principle. For the existence of $u$, we set
$$\mathcal B:=\{v \in PSH(V) \cap L^\infty(V): v^*|_{\partial V} \le \va, (dd^c v)^k \ge f\om^k\}.$$
Then $\mathcal B \ne \emptyset$ since $\psi_1 \in \mathcal B$ by the assumption.
Now we define
$$u(z):=\sup \{v(z): v \in \mathcal B\}, \  z \in V.$$
Since $V$ is locally irreducible, we have $u^* \in PSH(V) \cap L^\infty (V).$
For each $v \in \mathcal B$, by the hypothesis on $\psi_2$ we have
$$\varlimsup\limits_{z \to \xi} (v(z)+\psi_2(z)) \le 0, \ \forall \xi \in \partial V.$$
Thus, by the maximum principle we obtain $v+\psi_2<0$ on $V.$
It follows that $u+\psi_2 \le 0$ on $V$. By continuity of $\psi_2$ we obtain $u^* \le-\psi_2$ on $V.$
Hence  $u^*|_{\partial V} \le \va$.
Notice also that $u \ge \psi_1$ on $V$. In particular
$\varliminf\limits_{z\to \xi} u(z) \ge \va(\xi)$ for every $\xi \in \partial V.$
Therefore, $\lim\limits_{z \to \xi} u^*(z)=\va(\xi)$ for each $\xi \in \partial V.$
Now the main step is to show that $(dd^c u^*)^k=f\om^k$ on $V.$
Fix a point $z_0 \in V_{\reg}$. We can choose a \nhd\ $W$ of $z_0$ such that $U$ is biholomorphic to some ball in
$\C^k.$ By Lemma \ref{local}, we can find $\tilde u \in PSH(W)\cap L^\infty(W)$ such that
$$u^*|_{W} \le \tilde u, (dd^c \tilde u)^k=f\om^k, \lim\limits_{z \to \xi} \tilde u(z)=u^*(\xi), \
\forall \xi \in \partial W.$$
Then by the gluing lemma (cf. Lemma \ref{gluing}), the function
$$\hat u:=\begin{cases}&\tilde u\ \text{on}\ W\\
&u^*\ \text{on}\ V \setminus W\\
\end{cases}$$
belongs to $PSH(V) \cap L^\infty (V).$ Moreover $(dd^c \hat u)^k \ge f\om^k$ on $V$ by the choice of $u^*$, since the measure on the right hand side puts no mass on the "sphere" $\partial W.$
Thus by the definition of $u$ we have $u \ge \hat u$ on $V.$ This forces $u=\hat u=u^*$ on $U$. Consequently
$(dd^c u^*)^k=f\om^k$ on $U.$ Moreover, since $f\om^k$ does not charge $V_{\sig}$ a set of outer capacity zero (cf. Lemma \ref{outer}) we obtain that $(dd^c u^*)^k=f\om^k$ on $V$. By the choice of $u$, we must have $u=\hat u$ on $V$ and
$(dd^c u)^k=f\om^k$. The proof is thereby completed.
\end{proof}
\begin{remark}
If $V$ is not assumed to be locally irreducible then the Dirichlet problem is in general not solvable.
Recall the following simple example in [Wik2]. Consider the one dimensional complex variety $V:=\{(z,w) \in D: zw=0\},$
where $D$ is the unit ball in $\mathbb C^2.$
Let $\va: \partial V\to \mathbb R$ be defined by $\va (z,0)=0$ for $\vert z\vert=1$ and $\va (0,w)=1$ for
$\vert w\vert=1$ and $\mu=0.$ Suppose that there exists $u \in PSH(V)$ (which may not be locally bounded) such that
$dd^c u=0$ and $\lim\limits_{z\to \xi}u(z)=\va(\xi)$ for all $\xi \in \partial V.$
Then $u_1 (z):=u(z,0)$ and $u_2 (w):=u(0,w)$ are {\it harmonic} functions on the unit disk $\De$ of $\mathbb C$ with boundary values equal to $0$ and $1,$ respectively.
It follows that $u_1 \equiv 0$ and $u_2 \equiv 1$ on $\De,$ which is absurd.
\end{remark}
\begin{corollary}\label{weakversion}
Assume that $V$ is hyperconvex and locally irreducible. Let $f \ge 0$ be an upper-semicontinuous function with compact support on $V$. Then there exists a unique $u \in \E_0 (V)$ such that $(dd^c u)^k=f\om^k.$
\end{corollary}
\begin{proof}
By Remarks \ref{exhaustion1} and \ref{negative}, $V$ admits a $\mathcal C^\infty-$smooth negative strictly \psh\ exhaustion function $\rho$. Then for $\la>0$ large enough we have $(dd^c (\la\rho))^k \ge f\om^k$ on support of $f.$
Thus we may apply Theorem \ref{BT} with $\psi_1=\la \rho, \psi_2=0$ and $\va=0$ to reach the desired conclusion.
\end{proof}
The following result might be considered as an analogue of the classical Kolodziej's subsolution theorem for functions in $\E_1 (V).$ Our approach, however, is strongly inspired by Cegrell in [Ce3].
\begin{theorem} \label{Cegrelltheorem}
Let $V$ be a locally irreducible hyperconvex variety and $\mu$ be non-negative Radon measure on $V$
such that $\mu(V)>0.$ Suppose that there exists $v \in \E_1 (V)$ satisfying $(dd^c v)^k \ge \mu.$
Then there exists a unique $u \in \E_1 (V)$ such that $(dd^c u)^k=\mu.$
\end{theorem}
\begin{remark}
Let $E$ be a pluripolar subset of $V$ then $\mu(E)=0.$ Indeed, according to Proposition \ref{pluripolar}, there exists
$\va \in \F_1 (V)$ such that $\va|_E=-\infty.$
By the energy estimate given in Proposition \ref{classE1} (d) we obtain
$$\int\limits_V (-\va)d\mu \le \int\limits_V (-\va) (dd^c v)^k<\infty.$$
Thus $\mu(E)=0$ as desired.
\end{remark}
\begin{proof} The uniqueness of $u$ follows immediately from Corollary \ref{E1domination}. For the existence,
we first note that if $h \in \E_0 (V)$, then using again the energy estimate and the assumption we obtain
\begin{equation} \label{assumption}
\int\limits_V (-h)d\mu \le \int\limits_V (-h) (dd^c v)^k \le A\Big (\int\limits_V (-h)(dd^c h)^k\Big)^{\fr1{k+1}},
\end{equation}
where $A:= \Big (\int\limits_V (-v)(dd^c v)^k\Big)^{\frac k{k+1}}<\infty.$

Now we basically follow the scheme outlined in the proof of Theorem 3.2 in [Ce3].
There are two cases to be considered:

\noindent
{\it Case 1.} $\mu$ is compactly supported on $V_{\reg}$. We will show that there exists $u \in \F_1(V)$ such that
$(dd^c u)^k=\mu.$ Towards this end, we first
cover $K:=\su(\mu)$ by a finite number of open subsets $U_1, \cdots, U_s$ of $V_\reg$ such that for each $m=1, \cdots,s$, there exists a biholomorphic map $f_m: \mathbb B \to U_m$, where $\mathbb B$ is the unit ball in $\mathbb C^k.$
Let $\{\chi_1, \cdots, \chi_s\}$ be a partition of unity that subordinates to the covering $U_1, \cdots, U_s$
such that $\chi_m$ is non-vanishing on $K$ for every $m.$
For each $\ve>0,$ we also let $\rho_\ve$ be standard smoothing kernels with compact support in
$\mathbb B(0, \ve) \subset \mathbb C^k.$ Consider the following approximants to $\mu$
$$\mu^\ve:= \sum_m \chi_m \Big ((\nu_m*\rho_\ve)_*f_m\Big ),$$
where $\nu_m: =\mu^* f_m$ is the pull-back of $\mu$ under $f_m$.
It is then clear that $\mu^\ve$ are smooth functions with compact support in $V$. Moreover, for each {\it positive}
measurable function $\va$ defined on $V$ we have
\begin{equation} \label{convol}
\int_V \va d\mu =\sum_m \int_{\mathbb B} \va_m d\nu_m,
\int_V \va d\mu^\ve =\sum_m \int_{\mathbb B} (\va_m*\rho_\ve) d\nu_m,
\end{equation}
where $\va_m:= (\va\chi_m)\circ f_m.$ It follows that
$\mu^{\ve}$ converges weak$^*$ to $\mu$ as $\ve \downarrow 0.$
Now we let $\{\ve_j\}$ be an arbitrary sequence that decreases to $0$.
According to Corollary \ref{weakversion}, we can find $u_j \in \E_0 (V)$ such that
$(dd^c u_j)^k=\mu^{\ve_j}$ for every $j.$

\noindent
We now prove the following statements:

\noindent
(i) $\int\limits_V (-\va) d\mu=\lim\limits_{j \to \infty} \int\limits_V (-\va)d\mu^{\ve_j},
\forall \va \in \PSH^{-}(V)$ such that $\int\limits_V (-\va)d\mu<\infty.$

\noindent
(ii) There exists $j_0 \ge 1$ such that $2\int\limits_V (-u_j)d\mu \ge \int\limits_V (-u_j)d\mu_j$ for $j \ge j_0.$

\noindent
(iii) $u:=(\varlimsup\limits_{j \to \infty} u_j)^* \in PSH^{-} (V), \int\limits_V (-u)d\mu <\infty.$

\noindent
(iv) $\int\limits_V (-\va) d\mu \ge \int\limits_V (-\va)(dd^c u)^k, \forall \va \in \PSH^{-}(V).$

\noindent
(v) $\varlimsup\limits_{j \to  \infty} \int\limits_V (-v_j)(dd^c v_j)^k \le \varlimsup\limits_{j \to  \infty} \int\limits_V (-u_j)(dd^c u_j)^k \le \int\limits_V (-u)d\mu.$

\noindent
For (i), it suffices to show for each $m$ we have
$$\int\limits_{\mathbb B} (-\va_m)d\nu_m= \lim\limits_{j \to \infty} \int\limits_{\mathbb B} (-\va_m*\rho_{\ve_j})d\nu_m.$$
This follows from the following two facts which can be checked easily:

\noindent
(i') $(\va \circ f_m) *\rho_{\ve_j} \downarrow \va \circ f_m$ on $\mathbb B$ as $j \to \infty;$

\noindent
(ii') $\int\limits_{\mathbb B} (-\va_m)d\nu_m \le \int\limits_{V} (-\va)d\mu<\infty.$

\noindent
For (ii), it is enough to show that for every $1 \le m \le s$ there exist $j_0 \ge 1$ such that
$$2\int\limits_{\mathbb B} \eta_{j,m}*\rho_{\ve_j} d\nu_m \ge \int\limits_{\mathbb B} \eta_{j,m} d\nu_m \ \forall j \ge j_0,$$
where $\eta_{j,m}:=(u_j\chi_m)\circ f_m.$
This follows from (\ref{convol}) and the two facts that
$$u_j \circ f_m \le (u_j \circ f_m)*\rho_{\ve_j} \ \text{and}\ \chi_m \ne 0 \ \text{on}\ K.$$
\noindent
Next, we use (ii) and (\ref{assumption}) to get
$$\int\limits_V (-u_j)d\mu_j \le 2 \int\limits_V (-u_j)d\mu \le 2A\Big (\int\limits_V (-u_j)d\mu_j\Big )^{\fr1{k+1}}, \ \forall j \ge j_0.$$
This implies, for $j \ge j_0$ that
$$\int\limits_V (-u_j)d\mu_j \le 2 \sup\limits_j \int\limits_V (-u_j)d\mu \le A',$$
where $A'>0$ is a constant depends only on $A$ and $k.$
Hence $u:=(\varlimsup\limits_{j \to \infty} u_j)^* \in \PSH^{-} (V).$
Moreover, $v_j \downarrow u$ on $V$ where
$v_j:= (\sup\limits_{p \ge j} u_p)^* \in \E_0 (V),$ since $V$ is locally irreducible.
So by Lebesgue's monotone convergence theorem we obtain
$$\int\limits_V (-u)d\mu=\sup\limits_j\int\limits_V(-v_j)d\mu \le \sup\limits_j\int\limits_V (-u_j)d\mu<A'.$$
This is the assertion (iii).

\noindent
For (iv), we first use Proposition \ref{pro1} and Remark \ref{negative} to get a sequence
$\{\va_p\} \subset \E_0 (V)\cap \CC(\ov V)$ such that $\va_p \downarrow \va$ on $V.$
For each $p,$ by the monotonicity lemma (cf. Lemma \ref{monotonicity}) we get
$$\int_V (-\va_p) (dd^c u_j)^k \ge \int_V (-\va_p)(dd^c v_j)^k.$$
Now Theorem \ref{maintheorem1} implies that $(dd^c v_j)^k$ is weak$^*-$convergent to $(dd^c u)^k.$
So by letting $j \to \infty$ in the last inequality and taking into account the fact that $(-\va_p)(dd^c u_j)^k$ is
weak$^*-$convergent to $(-\va_p)\mu$ and all these measures vanish off a {\it fixed} compact set of $V$ we obtain
$$\int\limits_V (-\va_p) d\mu \ge \int\limits_V (-\va_p)(dd^c u)^k$$
Then it suffices to let $p \to \infty$ and use Lebesgue monotone's convergence theorem to get (iv).

\noindent
It remains to check (v). For this, we note that the inequality on the left hand side is a direct consequence of the monotonicity lemma (cf. Lemma \ref{monotonicity}). The proof of the second assertion is more intricate.
First we note that, by passing to a subsequence, we may achieve that $u_j$ converges to $u$ in $L^1_{\loc} (V).$
Now we argue by contradiction. Suppose that
$$\infty>a:=\varlimsup\limits_{j \to  \infty} \int\limits_V (-u_j)(dd^c u_j)^k>b:=\int\limits_V (-u)d\mu>0.$$
We choose $\de \in (0, a-b)$ such that
$a^{\fr{k-1}k} b<(a-\de)^k.$
Then by the same reasoning as in (ii) we can show, using (\ref{convol}) that there exists $j_0$ large enough such that for $j_0 \le j \le l$ we have
$$\int_V (-u_j)d\mu^{\ve_j} \le \int_V (-u_j)d\mu^{\ve_l}+\de$$
Then by using the integration by part formula (cf. Corollary \ref{coropartformular1}) and Lemma \ref{energy},
for $j_0 \le j \le l$ we have
$$\begin{aligned}
\int_V (-u_j)(dd^c u_j)^k &\le \int_V (-u_j)(dd^c u_l)^k +\de\\
&=\int_V (-u_l)dd^c u_j \wedge (dd^c u_l)^{k-1}+\de\\
& \le \Big (\int_V (-u_l)(dd^c u_j)^k\Big)^{\fr1{k}}\Big (\int_V (-u_l)(dd^c u_l)^k\Big)^{\fr{k-1}{k}}+\de.
\end{aligned}$$
Now we note that
$$\lim\limits_{l \to \infty} \int\limits_V (-u_l)(dd^c u_j)^k=
\lim\limits_{l \to \infty} \int\limits_V (-u_l)d\mu^{\ve_j}=\int\limits_V (-u)d\mu^{\ve_j}.$$
Here the last equality follows from (\ref{convol}) and the fact that
$\eta_{l, m}*\rho_{\ve_j}$ converges {\it uniformly} to $\eta_m*\rho_{\ve_j}$ as $l \to \infty$ on $\su{\nu_m}$
where
$\eta_{l,m}:=(u_l\chi_m)\circ f_m, \eta_m:= (u\chi_m)\circ f_m.$
This uniform convergent is a straightforward consequence of the fact that $u_l\chi_m$ converges to $u\chi_m$ in $L^1-$norm on the support of $\chi_m.$
Putting all these together we obtain
$$\int\limits_V (-u_j)(dd^c u_j)^k \le a^{\fr{k-1}k}\Big (\int\limits_V (-u)d\mu^{\ve_j} \Big)^{\fr1k}+\de.$$
By letting $j \to \infty$ and using (ii) we reach
$$a=\varlimsup\limits_{j \to \infty} \int\limits_V (-u_j)(dd^c u_j)^k \le a^{\fr{k-1}k}b^{\fr1k}+\de.$$
After rearranging we arrive at a contradiction to the choice of $\de.$ Thus we have proved (iv).
The rest of our proof goes as follows.
Let $\theta \in \E_0 (V)$ be given. For $t \ge 0$, from (ii) and (iv) we obtain
$$\begin{aligned}
\int\limits_V -(u+t\theta)d\mu&=\lim_{j \to \infty}\int\limits_V -(u+t\theta)(dd^c u_j)^k\\
&\le \Big (\int\limits_V -(u+t\theta)(dd^c (u+t\theta))^k\Big )^{\fr1{k+1}}
\Big (\varlimsup\limits_{j \to \infty} \int\limits_V (-u_j)(dd^c u_j)^k\Big)^{\fr{k}{k+1}}\\
&\le \Big (\int\limits_V (-u)d\mu\Big)^{\fr{k}{k+1}}
\Big (\int\limits_V -(u+t\theta)(dd^c u+tdd^c \theta)^k\Big )^{\fr1{k+1}}.
\end{aligned}$$
It follows that $P(t) \le Q(t)$ for $t \ge 0,$ where
$$P(t):= \Big (t \int\limits_V (-\theta)d\mu+\int\limits_V (-u)d\mu\Big )^{k+1},
Q(t):= \Big (\int\limits_V (-u)d\mu\Big)^k\Big (\int\limits_V -(u+t\theta)(dd^c u+tdd^c \theta)^k\Big ).$$
Put $t=0$ we have
$$\int\limits_V (-u)d\mu=P(0)\le Q(0)=\int\limits_V (-u)(dd^c u)^k.$$
Combining this with (i)  we obtain
$$\int\limits_V (-u)d\mu=\int\limits_V (-u)(dd^c u)^k.$$
This means that $P(0)=Q(0).$ Hence $P'(0^{+}) \le Q' (0^{+}).$ By some easy computations and using Corollary \ref{coropartformular1} we have
$$\begin{aligned}
P'(0^{+})&=(k+1) \Big (\int\limits_V (-u)d\mu \Big )^k \int\limits_V (-\theta)d\mu,\\
Q'(0^{+})&= \Big (\int\limits_V (-u)d\mu\Big)^k \Big (\int_V (-\theta) (dd^c u)^k+k\int_V (-u)(dd^c u)^{k-1} \wedge dd^c \theta\Big)\\
& \le (k+1) \Big (\int\limits_V (-u)d\mu\Big)^k\int\limits_V (-\theta)(dd^c u)^k.
\end{aligned}$$
It follows, using $\int\limits_V (-u)d\mu>0$, that
$$\int\limits_V (-\theta)d\mu \le \int\limits_V (-\theta)(dd^c u)^k.$$
Taking (i) into account again we actually have
$$\int\limits_V (-\theta)d\mu=\int\limits_V (-\theta)(dd^c u)^k.$$
Since this identity holds true for every $\theta \in \E_0 (V),$ by Lemma \ref{E0} we have $(dd^c u)^k=\mu$.
Observe also that, since $\sup\limits_j\int\limits_V (dd^c v_j)^k \le \sup\limits_j\int\limits_V (dd^c u_j)^k <\infty$
and since $v_j \downarrow u,$ using (v) and Theorem \ref{maintheorem1} we conclude that $u \in \F_1 (V)$.
Recall that $\int\limits_V (-u)d\mu<A'$ where $A>0$ depends only on $k, A.$
The proof is complete in this special case.

\noindent
{\it Case 2.} $\mu$ is a general non-negative Radon measure.
Let $0 \le \chi_p \uparrow 1$ be a sequence of smooth function with compact support in $V_{\reg}.$
By the remark that follows Theorem \ref{Cegrelltheorem}, $\mu$ does not charge $V_{\sig}.$
Therefore $\mu_p:=\chi_p \mu$ converges weak$^*-$ to $\mu.$
By Case 1, we can find $u_p \in \F_1 (V)$ such that
$(dd^c u_p)^k=\mu_p.$
Moreover, by the domination principle (cf. Theorem \ref{E1domination}),
for each $p$ we have $u_p \ge u_{p+1} \ge v.$
Therefore $u_p \downarrow u \in \E_1 (V)$ on $V,$ by Proposition \ref{classE1}.
Finally, using Theorem \ref{maintheorem2} we have $(dd^c u)^k=\mu.$
The proof is thereby completed.\end{proof}
\vskip1cm
\centerline{\bf References}

\noindent
[Be] E. Bedford, {\it The operator $(dd^c)^n$ on complex spaces}, S\'eminaire Lelong-Skoda, Springer Lecture Notes
{\bf 919} (1981), 294-323.

\noindent
[BL] T. Bloom and N. Levenberg, {\it Distribution of nodes on algebraic curves in $\mathbb C^N$,} Ann. Inst. Fourier (Grenoble) {\bf 53} (2003), 1365-1385.

\noindent
[B\l] Z. B\l ocki, {\it The complex Monge-Amp\`ere operators in hyperconvex domains,} Annali della Scuola Normale Superiore di Pisa {\bf 23} (1996), 721-747.

\noindent
[BT1] E. Bedford and A. Taylor, {\it The Dirichlet problem for the complex Monge-Amp\`ere equation},
Invent. Math. {\bf 37} (1976), 1-44.

\noindent
[BT2] E. Bedford and A. Taylor, {\it A new capacity for \psh\ functions,} Acta. Math. {\bf 149} (1982), 1-40.

\noindent
[BT3] E. Bedford and A. Taylor, {\it Fine topology, Shilov boundary, and $(dd^c)^n,$} Journal of Functional Analysis
{\bf 72} (1987), 225-251.

\noindent
[Ce1] U. Cegrell, {\it Pluricomplex energy}, Acta Math. {\bf 180} (1998), 187-217.

\noindent
[Ce2] U. Cegrell, {\it The general definition of the complex Monge-Amp\`ere operator}, Ann.
Inst. Fourier (Grenoble) {\bf 54} (2004), 159-179.

\noindent
[Ce3] U. Cegrell, {\it Measures of finite pluricomplex energy}, http://front.math.ucdavis.edu/1107.1899.

\noindent
[Ce4] U. Cegrell, {\it Convergence in capacity}, Canad. Math. Bull., {\bf 55} (2012), 242-248.

\noindent
[Ch] E. Chirka, {\it Complex Analytic Sets}, Kluwer, Dordecht, 1989.

\noindent
[CP] U. Cegrell and L. Persson, {\it An energy estimate for the complex Monge-Amp\`ere operator}, Ann. Polon. Math.
{\bf 67} (1997), 95-102.

\noindent
[De] J.-P. Demailly, {\it Mesures de Monge-Amp\`ere et caract\'erisation g\'eom\'etrique des vari\'et\'es
alg\'ebriques affines}, M\'em. Soc. Math. France (N.S.) {\bf 19} (1985) 124 pp.

\noindent
[DLS] N. Q. Dieu and T.V. Long, {\it Approximation of plurisubharmonic functions on complex varieties},
https://arxiv.org/abs/1611.03577.

\noindent
[DS] N.Q. Dieu and O. Sanphet, {\it A comparison principle for bounded \psh\ functions on complex varieties of
$\mathbb C^n$}, https://arxiv.org/abs/1702.06732.

\noindent
[FN] J. E. Fornaess and R. Narasimhan, {\it The Levi problem on complex spaces with singularities},
Math. Ann., {\bf 248} (1980) 47-72.

\noindent
[Pe] L. Persson, {\it A Dirichlet principle for the complex Monge-Amp\`ere operator}, Ark. Mat., {\bf 37} (1999),
345-356.

\noindent
[Siu] Y. T. Siu, {\it Every Stein subvariety admits a Stein neighborhood,} Invent. Math., {\bf 38} (1977), 89-100.

\noindent
[Wik1] F. Wikstr\"om, {\it Continuity of the relative extremal function on analytic varieties in $\mathbb C^n,$}
Ann. Polon. Math. {\bf 86} (2005), 219-225.

\noindent
[Wik2] F. Wikstr\"om, {\it The Dirichlet problem for maximal \psh\ functions on analytic varieties}, Int. Journal of Math. {\bf 20} (2009), 521-528.

\noindent
[Ze] A. Zeriahi, {\it Fonction de Green pluricomplexe \`a p\^ole à l'infini sur un espace de Stein parabolique et applications,}  Math. Scand. {\bf 69} (1991), no. 1, 89-126.
\end{document}